\documentclass[11pt,a4paper,twoside,dvipsnames,table]{article}

\usepackage{fullpage}
\usepackage[utf8]{inputenc}
\usepackage[T1]{fontenc}
\usepackage{float}
\usepackage[abbrvbib, preprint]{jmlr2e}

\usepackage{authblk}        % authors affiliation
\usepackage[normalem]{ulem} % strike text
\usepackage{bm} 								% For bold font in math environment
\usepackage{mathtools}							% For \MoveEqLeft, \coloneqq, etc.
\usepackage{xspace}							% To handle spaces in commands.
\usepackage{nicefrac}                                                  
\usepackage{paralist}
\usepackage{ifthen}
\newboolean{pictures}
\setboolean{pictures}{false}

\usepackage{graphicx}
\usepackage{subcaption}
\usepackage{tikz}
\usepackage{pgfplots}
\usepackage{caption}
\usepackage{xcolor} % colors for tabular rows
\mathtoolsset{showonlyrefs} % show refs only when referred to
\usepackage{dsfont}              % for black board digits (\mathds{1})

\usepackage{algorithm, algcompatible}
\renewcommand{\COMMENT}[2][0.15\linewidth]{%
  \leavevmode\hfill\makebox[#1][l]{//~#2}}
\algnewcommand\algorithmicto{\textbf{to}}
\algnewcommand\RETURN{\State \textbf{return} }

\usepackage{tikz}
\DeclareMathOperator*{\argmax}{arg\,max}
\DeclareMathOperator*{\argmin}{arg\,min}

\newcommand{\e}{\varepsilon}

\newcommand{\dd}{\ensuremath{\mathrm d}}
\newcommand{\ds}{\ensuremath{\mathrm ds}}
\newcommand{\dt}{\ensuremath{\mathrm dt}}

\newcommand{\dx}{\ensuremath{\mathrm dx}}

\newcommand{\cC}{\ensuremath{\mathcal{C}}}

\newcommand{\cI}{\ensuremath{\mathcal{I}}}
\newcommand{\cJ}{\ensuremath{\mathcal{J}}}

\newcommand{\cL}{\ensuremath{\mathcal{L}}}

\newcommand{\cN}{\ensuremath{\mathcal{N}}}

\newcommand{\cS}{\ensuremath{\mathcal{S}}}
\newcommand{\cT}{\ensuremath{\mathcal{T}}}

\newcommand{\cV}{\ensuremath{\mathcal{V}}}

\newcommand{\bE}{\ensuremath{\mathbb{E}}}

\newcommand{\bN}{\ensuremath{\mathbb{N}}}

\newcommand{\bP}{\ensuremath{\mathbb{P}}}

\newcommand{\bR}{\ensuremath{\mathbb{R}}}

\newcommand{\bp}{\ensuremath{\bm{p}}}
\newcommand{\bu}{\ensuremath{\bm{u}}}

\renewcommand{\u}{\ensuremath{u}}

\newcommand{\bphi}{\ensuremath{{\bm{\varphi}}}}
\newcommand{\btheta}{\ensuremath{{\bm{\theta}}}}
\newcommand{\bvtheta}{\ensuremath{{\bm{\vartheta}}}}
\newcommand{\F}{\ensuremath{{\bm{F}}}}

\newcommand{\st}{\ensuremath{L^\infty([0,T], \Theta)}}
\renewcommand{\sp}{\ensuremath{L^\infty([0,T], \bR^n)}}

\newcommand{\wH}{\ensuremath{\widetilde H}}

\newcommand{\fwd}{\texttt{solve\_fwd}\xspace}
\newcommand{\bwd}{\texttt{solve\_bwd}\xspace}

\newcommand{\MSA}{\text{MSA}\xspace}
\newcommand{\AMSA}{\text{A-MSA}\xspace}
\newcommand{\EMSA}{\text{E-MSA}\xspace}

\newcommand{\loss}{\ensuremath{\text{Loss}}}

\newcommand{\kmax}{\ensuremath{k}_\textrm{max}}

\usepackage{hyperref}
\hypersetup{
 pdfauthor={Joubine Aghili and Olga Mula},
 pdftitle={Adaptive Deep Neural Networks},
 pdfsubject={},
 pdfkeywords={Deep Neural Networks; Learning; Data Assimilation}
}

\begin{document}

\title{
Depth-Adaptive Neural Networks \\ from the Optimal Control viewpoint
%\om{[1] Shallow-to-Deep training of Deep Neural Networks \\based on Optimal Control}
}

\author[1]{Joubine Aghili}
\author[1, 2]{Olga Mula\thanks{This research was supported by the Emergence Project ``Models and Measures'' of the Paris City Council, \texttt{\{aghili,mula\}@ceremade.dauphine.fr}}}
 
\affil[1]{CEREMADE, CNRS, UMR 7534, Université Paris-Dauphine, PSL University, 75016 Paris, France}
\affil[2]{Inria Paris, Commedia Team, 2 rue Simone Iff 75012 Paris}

% \author{Joubine Aghili and Olga Mula\thanks{%
%   
%   }
% }
\date{}
\maketitle

\begin{abstract}
In recent years, deep learning has been connected with optimal control as a way to define a notion of a continuous underlying learning problem. In this view, neural networks can be interpreted as a discretization of a parametric Ordinary Differential Equation which, in the limit, defines a continuous-depth neural network. The learning task then consists in finding the best ODE parameters for the problem under consideration, and their number increases with the accuracy of the time discretization. Although important steps have been taken to realize the advantages of such continuous formulations, most current learning techniques fix a discretization (i.e.~the number of layers is fixed). In this work, we propose an iterative adaptive algorithm where we progressively refine the time discretization (i.e.~we increase the number of layers). Provided that certain tolerances are met across the iterations, we prove that the strategy converges to the underlying continuous problem. One salient advantage of such a shallow-to-deep approach is that it helps to benefit in practice from the higher approximation properties of deep networks by mitigating over-parametrization issues. The performance of the approach is illustrated in several numerical examples.

\end{abstract}


%%%%%%%%%%%%%%%%%%%%
\section{Introduction}
\subsection{Context}
Neural networks produce structured parametric families of functions that have been studied for at least 70 years (see \cite{Hebb1949,Rosenblatt1958}). It is however only in the last decade that their popularity has surged. Thanks to the increase of computing power, and the development of easy-to-use computing tools for optimization and automatic differentiation, statistical learning of neural networks has produced state-of-the-art performance in a large variety of machine learning problems, from computer vision \cite{KSH2012} (e.g.~self-driving cars, X-rays diagnosis,...) to natural language processing \cite{WSCLNMKCG2016} (e.g.~Google translate, DeepL Translator,...) and reinforcement learning (e.g.~superhuman performance at Go \cite{GO-game-2016}). Despite this great empirical success, neural networks are not entirely well-understood and there is a pressing need to provide:
\begin{itemize}
\item solid \emph{mathematical foundations} to understand their approximation power and why they may outperform other classes of functions,
\item \emph{algorithms} to find systematically an appropriate architecture for each problem and to train the network in order to deliver in practice the approximation capabilities predicted by the theory, and to guarantee robustness over generalization errors.
\end{itemize}
On the first point, significant advances have been made recently on the approximation power of neural networks (see, e.g., \cite{Yarotsky2017,DDFHP2019, GPEB2019, GKP2019} for a selection of rigorous results). In particular, a few recent works have given theoretical evidence on the advantages of using deep versus shallow architectures for the approximation of certain relevant families of functions (see, e.g., \cite{Telgarsky2015,DDFHP2019}). However, so far the obtained results do not seem to be informative on how to address the second point above, that is, how to build algorithms that discover automatically the right architecture for each problem, and that may allow to benefit from the theoretically high approximation properties of (deep) neural networks. As a result, their training remains a key open issue subject to very active research. The present work is a contribution in this direction, with special focus on understanding the underlying mechanisms of training efficiently deep architectures.

The most commonly applied training methods are based on the stochastic gradient descent (\cite{RM1951, Bottou2010}) and its variants (see, e.g., \cite{DHS2011, Zeiler2012}). It has the advantage of being easily implementable but it is difficult to tune properly in practice, especially in problems involving large data sets and using deep neural networks with many layers and coefficients which are very prone to over-parametrization issues. The difficulty in training  deep networks raises the question regarding the benefits of applying a \emph{shallow-to-deep adaptive strategy} in the network architecture (and a hence coarse-to-fine adaptive strategy in the number of coefficients) in order perform the training in a more robust and rapid manner. The underlying intuition is that shallow networks give poor approximation properties but converge faster than deeper networks which, on the contrary, have higher approximation power.

There is a large body of emerging works which explore numerically the potential of training deep neural networks with shallow-to-deep strategies (see, e.g., \cite{CGS2016, WWRC2016, LH2017, CGKMY2017, EMH2018}). In this work, we present an abstract setting and convergence results which are a first theoretical justification of this type of algorithmic approach. In order to understand the potential gain, it is  necessary to define a notion of the continuous underlying objects that are approximated. This is the reason why we have chosen to work from a perspective which connects the task of learning with neural networks with optimal control and dynamical ODE systems. In this view, there is a notion of a continuous-depth neural network which is described by a continuous time-dependent ODE system. Its discretization fixes its architecture. For instance, ResNet \cite{HZRS2016} can be regarded as an Explicit Euler scheme for the solution of certain dynamics. The weights involved in the ODE are seen as controls and correspond to the parameters of the network which have to be optimized. The process of learning these parameters can then be recast as an optimal control problem over the admissible controls where the cost function is the empirical risk (sometimes with an extra regularization term). Necessary optimality conditions are then easily formulated through the Pontryagin Maximum Principle (PMP, \cite{BGP1960, Pontryagin2018}).

\subsection{Contribution of the present work}
The optimal control point of view is appealing since it gives access to a fully continuous description of neural networks and of the underlying statistical learning process. The practical implementation involves two discretization errors:
\begin{enumerate}
\item \textbf{Sampling error:} The cost function of the fully continuous optimal control problem is the average of a loss function, sometimes also called risk. However, in practice, the average loss is replaced by an empirical mean which is built from $N$ samples. So it is necessary to assess by how much the minimum the sampled mean deviates from the continuous underlying problem. We analyze this point in Theorem \ref{th:impact-sampling}.
\item \textbf{Time-integration error:} To solve the optimal control problem on the empirical mean of the loss, one solves the resulting PMP with Newton-type iterations involving the solution of forward and backward dynamical systems with certain time integration schemes (which fix the neural network architecture). The usual approach is to first fix a discretization and then solve the discrete version of the PMP. However, the discretization errors accumulate at each iteration and make us deviate from the time-continuous version. It is thus necessary to examine how much the final output deviates from the time-continuous one.
\end{enumerate}
In this context, the main contributions of the present work are:
\begin{enumerate}
\item \textbf{Derivation of a full error analysis:}
\begin{itemize}
\item We make a theoretical analysis on the interplay between the above two errors. The impact of the sampling error is analyzed in Theorem \ref{th:impact-sampling} and the discretization error in Theorem \ref{th:conv-AMSA}. From these results, it follows that to approximate the minimum of the underlying fully continuous problem at a certain accuracy, we not only need a minimal number of training samples, but we also need to solve the control problem on the empirical risk with some minimal accuracy. This requires, in turn, to use time-adaptive techniques. In our analysis, both errors are additive and can actually be optimally balanced. The final result is given in Corolary \ref{cor:final-conv-result} and it is formulated in probability.
\item To solve the control problem on the empirical risk at a given target accuracy, we show that we have to compute the forward and backward propagations at each step of the Newton-type algorithms at increasing accuracy. As our theoretical proof of Theorem \ref{th:conv-AMSA} will illustrate, tightening the accuracy of the time integration is necessary to guarantee convergence to the exact continuous problem and justifies the coarse-to-fine, resp.~shallow-to-deep, strategy. This yields neural networks with depths that are automatically adapted.
\end{itemize}
\item \textbf{Practical coarse-to-fine adaptive algorithm:}
\begin{itemize}
\item We next use the theory to derive actionable criteria to build time-integration schemes that are adaptively refined across the iterations. This has the advantage of reducing computational cost since early iterations use looser tolerances, thus avoiding unnecessary work due to oversolving, while later iterations use tighter tolerances to deliver accuracy. The coarse-to-fine strategy could also be understood as a dynamic regularization which helps to prevent over-parametrization.
\item We illustrate the behavior of the algorithm in several benchmark examples.
\end{itemize}
\end{enumerate}

\subsection{Potential impact, limitations and extensions}
We believe the results of this article can contribute to the following key topics in deep learning:
\begin{itemize}
\item Automate the selection of neural network architecture during the training,
\item Enhance the interpretability of the generalization errors by connecting them with sampling and discretization errors,
\item Provide a first theoretical justification of the works adopting the principle of shallow-to-deep training of deep neural networks.
\end{itemize}
Some limitations and natural questions that arise for future works are the following:
\begin{itemize}
\item Our theoretical results involve numerous bounds and some of them might be suboptimal. Certain constants involved are also difficult to estimate in practice. As a result, the theoretical convergence result cannot be used to derive fully explicit adaptive strategies for implementation. However it gives hints on how to proceed in practice, and we propose practical guidelines to perform adaptivity. In future works, it would be desirable to find a sharper analysis with constants that are easier to compute in order to obtain a fully closed pipeline. This may however require very different arguments as our current ones.
\item As we will explain in the next section, the optimal control point of view only describes certain classes of neural networks since it introduces certain restrictions in the architecture. As a result, our adaptive setting cannot automate the training of all existing classes of neural networks. So a natural question is how to extend the current ideas to other training techniques that are not based on optimal control.
\item One important question is whether the flow map of a given dynamical system can represent the data at hand. The same issue also arises in the classical approach of deep learning in terms of the optimality of the choice of the network. From the optimal control perspective, this representability problem can be viewed as a controllability problem. It would be interesting to transfer controllability results to the current deep learning framework to address this question.
\end{itemize}

\subsection{Outline of the paper}
In section \ref{sec:deep-learning}, we define precisely what we understand by deep learning and we recall its connection with optimal control. Section \ref{sec:optimal-control} gives the optimal control setting which we use in our subsequent developments. In particular, we formulate the fully continuous optimal control problem over the expectation of the loss and its sampled version involving an empirical mean of the loss. Section \ref{sec:PMP-EMSA} recalls the Pontryagin's Maximum Principle for the sampled problem. We also recall in this section the Extended Method of Successive Approximations (\EMSA), originally introduced in \cite{LCTE2018}, to solve the sampled problem. The algorithm is iterative, and it requires  to compute at each iteration a forward and a backward time propagation followed by a certain maximization over the control variables. In its original formulation, it is assumed that these steps are performed exactly. In Section \ref{sec:EMSA}, we formulate our main algorithm, which is an inexact version of the \EMSA that we call Adaptive MSA (\AMSA). It is based on approximately realizing each propagation and maximization within a certain accuracy. To ensure convergence to the exact solution, the accuracy needs to be tightened across the iterations. In section \ref{sec:conv-AMSA}, we prove that \AMSA converges to a time-continuous solution of the sampled PMP. In section \ref{sec:sampling-err}, we connect the solution of the sampled problem with the continuous one involving the exact average of the loss function. For this, we use results obtained in \cite{EHL2019}. This last step allows to connect how much the solution of \AMSA deviates from the fully continuous one in terms of the number of samples and the accuracy of the discretization. In section \ref{sec:impl}, we give guidelines to implement \AMSA in practice, and section \ref{sec:numerics} illustrates the performance of the algorithm in numerical examples. We conclude the paper in Section \ref{sec:conclusion}.

%This approach has gained increasing attention over the past years and most algorithms rely then on solving the necessary Pontryagin Maximum Principle (PMP, \cite{}) with the Method of Successive Approximations (MSA, \cite{}). It is an iterative algorithm that requires to solve at each step a forward and a backward time evolution, followed by a maximization step. The usual approach is to fix a time discretization and solve the problem on this grid without, however, being able to actually quantify how far the final output is with respect to the solution of the continuous PMP problem.

%\om{[For this, we start from the continuous formulation of MSA where every step is assumed to be performed exacty (and thus cannot be computed in practice). The feasible adaptive scheme is then based on approximately realizing this iteration within dynamically updated accuracy tolerances that still ensure convergence to the exact solution. To ensure that the stage dependent error tolerances are met, one has to employ rigorous a posteriori error bounds. The final output of this procedure is a specific discretization of a forward time evolution, which can be interpreted as a neural network whose architecture has been adapted to yield results that can be made arbitrarily close to the continuous control problem.}

\section{From deep learning to optimal control}
\label{sec:deep-learning}
\subsection{Statistical Learning Problems}
\label{sec:stat-learning}
We consider the following \emph{statistical learning problem}: Assume that we are given a \emph{domain set} $X\subseteq \bR^n$ and a \emph{label set} $Y\subseteq \bR^k$, with $n,\, k\in \bN$. Further assume that there exists an unknown probability distribution $\mu$ on $X\times Y$. Given a loss function $\cL:Y\times Y \to \bR^+$, the goal of the statistical learning problem is to find a function $v:X\to Y$, which we will call \textit{prediction rule}, from a hypothesis class  $\cV \subset \{v:X \to Y\}$ such that the expected loss
$$
\cJ(v) \coloneqq \bE_{(x, y) \sim \mu} \cL(v(x), y)
$$
is minimized over $\cV$. In other words, the task is to find
$$
v^* \in \argmin_{v\in \cV} \cJ(v).
$$
In general, the probability distribution $\mu$ is unknown and we are only given a set $\cS_N$ of $N\in\bN$ training samples
$$
\cS_N \coloneqq \{(x_i, y_i)\}_{i=1}^N.
$$
The most common choice is then to consider the uniform distribution $\mu_N : X\times Y \to \bR^+$,
$$
\mu_N (x, y) = \frac 1 N \sum_{i=1}^N \delta_{(x_i, y_i)}(x,y)
$$
which yields the so-called \emph{empirical loss}
$$
\cJ_N(v) \coloneqq \bE_{(x, y) \sim \mu_N} \cL(v(x), y) = \frac 1 N \sum_{i=1}^N \cL(v(x_i), y_i)
$$
and one finds $v_N$, an approximation of $v$, by minimizing over  it,
\begin{equation}
\label{eq:learning-general}
v_N \in \argmin_{v\in \cV} \cJ_N(v).
\end{equation}
In the following, the above optimization step will be called the \emph{learning procedure}.

\subsection{Neural Network Architectures and Deep Learning}
While there exists numerous different architectures, we could simplify the discussion and say that neural networks are a class of functions with the basic general form
\begin{equation}
\label{eq:NNgeneral}
v:\bR^n \to \bR^k,\quad x \mapsto W_L\sigma( W_{L-1}\sigma(\dots( \sigma(W_1(x)) )  ),
\end{equation}
where:
\begin{itemize}
\item $L\in \bN$ is the \textit{depth}, i.e., the number of layers of the neural network,
\item $W_\ell:\bR^{N_{\ell-1}}\to\bR^{N_\ell}$ are affine maps for $\ell=1,\dots, L$. For consistency, we must set $N_0=n$ and $N_L = k$ but the rest of the dimensions $N_\ell\in \bN$ can be freely chosen. We have for all $x\in\bR^{N_{\ell-1}}$, $W_\ell(x)=A_\ell(x) + b_\ell$ for a matrix $A_\ell \in \bR^{N_\ell\times N_{\ell-1}}$ and a vector $b_\ell \in \bR^{N_\ell}$.
\item $\sigma: \bR \to \bR$ is a \emph{nonlinear activation function} which is applied coordinate-wise in \eqref{eq:NNgeneral}. Some popular choices are the ReLU function, $\sigma(x) = \max(0,x)$, or the hyperbolic tangent, $\sigma(x) = \tanh(x)$.
\end{itemize}
\emph{Deep learning} describes the range of learning procedures to solve statistical learning problems where the hypothesis class $\cV$ is taken to be a set of neural networks. One usually works with the class of neural networks with given depth $L$, activation function $\sigma$ and fixed dimensions $N_1, \dots, N_{L-1}$ for the affine mappings,
$$
\cN\cN(L, \sigma, N_1, \dots, N_{L-1}) := \{ v: X \to Y \,:\, v(x) = W_L\sigma( W_{L-1}\sigma(\dots( \sigma(W_1(x)) )  )  \}.
$$
For this class, the task is to solve the empirical risk problem \eqref{eq:learning-general} associated to it.

\subsection{Continuous Formulation as an Optimal Control Problem}
\label{sec:cont-form}
Functions of the type \eqref{eq:NNgeneral} can be built by repeated composition of  parametrizable functions
$$
\phi_\ell : \bR^{N_{\ell-1}}\to \bR^{N_\ell},\quad x\mapsto \phi_\ell(x) = \sigma(W_\ell(x)),\quad i =1,\dots, L-1,
$$
followed by an affine step, that is,
\begin{equation}
v = W_L \circ \phi_{L-1}\circ \dots \circ \phi_1, \quad\forall v \in \cN\cN(L, \sigma, N_1, \dots, N_{L-1}).
\end{equation}
This simple observation yields to numerous other possible architectures. The most relevant instance triggering the connection with optimal control are the so-called Residual Neural Networks (ResNet, \cite{HZRS2016}). They correspond to the choice
$$
\phi_\ell : \bR^{N_{\ell-1}}\to \bR^{N_\ell},\quad x\mapsto \phi_\ell(x) = x + h \sigma(W_\ell(x)),\quad i =1,\dots, L-1,
$$
for a certain parameter\footnote{In fact, ResNets were originally defined for $h=1$. The ``augmented'' definition with the paramater $h$ was introduced in \cite{HR2018}.} $h>0$. 

If we now set $N_{L-1}=\dots = N_1 = N_0 = n$, the application of $\phi_\ell$ can be interpreted as an Explicit Euler step from time $t_{\ell-1}= (\ell-1) h$  to time $t_{\ell} = t_{\ell-1} +h$ of the dynamics
$$
\dot x_t = f(x_t, \theta_t),\quad \text{ with } x_0 = x,
$$
where
$$
f(x_t, \theta_t) = \sigma( W_t (x_t)),\quad \text{with }W_t(x) = A_t x + b_t,
$$
and $\theta_t$ gathers the parameters upon which the dynamics depend. In our case,
$$
\theta_t=(A_t, b_t).
$$
Note that the matrices $A$ and $b$ are now time-dependent, hence the $t$-subscript.

In this view, ResNet functions can be interpreted as the output of performing $L-1$ time steps of size $h$ of the above dynamics, followed by an affine operation $W_L:\bR^n \to \bR^k$. This last operation is important since it maps the final output from the domain set in $\bR^n$ to the label set in $\bR^k$. In fact, it can be replaced by any linear or nonlinear mapping $g:\bR^n \to \bR^k$. Note also that $h$ can be included as a parameter of the model class and that one can optimize over it as in \cite{BCEOS2019}.  

Following similar lines, one can view certain classes of neural networks as discretizations of an underlying continuous dynamical system. Table \ref{tab:DNN-list} lists some popular classes and the associated numerical scheme (see, e.g., \cite{LZLD2018} for more details on these connections). We  emphasize however that the connection with optimal control comes at the price of imposing that the input and intra-layer dimensions are equal ($N_{L-1}=\dots = N_1 = N_0$), thus it cannot be made for all classes of neural networks.

\begin{table}
\begin{center}
  \begin{tabular}{ | l | l |}
    \hline
    Network Architecture & Associated ODE discretization \\
    \hline
    ResNet \cite{HZRS2016}, RevNet \cite{GRUG2017} & Forward Euler \\
    \hline
    Polynet \cite{ZLCL2017}  & Approximation of Backward Euler \\
    \hline
    FractalNet \cite{LMS2016}  & Runge-Kutta order 2\\
    \hline
  \end{tabular}
\end{center}
\caption{Some neural network classes and their associated ODE discretization.}
  \label{tab:DNN-list}
\end{table}

\subsection{Related works}
To the best of our knowledge, the connection between deep learning, dynamical systems and optimal control has been known since at least the 1980's, going back to the works of LeCun and Pineda in which the idea of back-propagation is connected to the adjoint variable arising in optimal control (see \cite{Lecun1988, Pineda1988}). The approach has  gained increasing attention in recent years. We refer to \cite{E2017, HR2018, LCTE2018, BCEOS2019, VKWN2020} for some selected references. The \EMSA deep learning algorithm based on the optimal control framework has been introduced in \cite{LCTE2018} and is the main starting point for our developments on an adaptivity. A recent development on the topic of adaptivity has been made in \cite{BCEOS2019} where the idea is to fix a time-scheme with a \emph{fixed} number of time steps, and include the time step as an additional parameter to be optimized. Note that this type of adaptivity lies in the time step sizes, and it is different than the one proposed here since we understand adaptivity as refining the time-discretization, thus increasing progressively the number of time steps.

\section{Optimal control setting}
\label{sec:optimal-control}
We next define the notation for the optimal control setting which we will use in the rest of the paper. The notation is kept as consistent as possible with section \ref{sec:deep-learning} in order to further enhance the similarity between the usual deep learning approach and the present one.

Like in section \ref{sec:deep-learning}, we assume that we are given a \emph{domain set} $X\subset \bR^n$ and a \emph{label set} $Y\subset \bR^k$, with $n,\, k\in \bN$, and that there exists an unknown probability distribution $\mu$ on $X\times Y$ representing the distribution of the input-target pairs $(x,y)$. Consider now a set of admissible controls or training weights $\Theta \subseteq \bR^m$. Usually, we set $\Theta = \bR^m$ in deep learning but the present methodology allows to consider constraints we will denote $\Theta \subseteq [U_{\mathrm{min}},U_{\mathrm{max}}]^m$.
Fix $T>0$ and let $f$ (feed-forward dynamics), $\Phi$ (terminal loss function) and $R$ (regularizer) be functions
$$
f:\bR^n \times \Theta \to \bR^n,\quad \Phi:\bR^n\times \bR^k \to \bR, \quad R: \Theta \to \bR.
$$
Let $\st$ be the set of essentially bounded measurable controls. In the following, we use bold-faced letters for path-time quantities. For example, $\btheta \coloneqq \{ \theta_t \,:\, 0\leq t \leq T \}$ for any $\btheta\in\st$. For every control $\btheta\in\st$ and every value of the random variable $x \in X$, we define the state dynamics $\bu^{\btheta, x} \coloneqq \{ \u_t^{\btheta, x } \,:\, 0\leq t \leq T \}$ as the solution to the ordinary differential equation (ODE),
\begin{equation}
\label{eq:ODE}
\begin{cases}
\dot{ \u}^{\btheta, x}_t &= f(\u^{\btheta, x}_t, \theta_t),\quad \forall t\in (0,T) \\
\u^{\btheta, x}_0 &= x
\end{cases}
\end{equation}
The ODE is stochastic and its only source of randomness is the initial condition $x$.

With this notation, the deep learning optimization problem can be posed as the optimal control problem of finding
\begin{equation}
\label{eq:oc} % renamed from eq:ocs to eq:oc
J^* = \inf_{\btheta\in\st}\cJ(\btheta) ,\quad \text{subject to \eqref{eq:ODE}},
\end{equation}
where
\begin{equation}
\cJ(\btheta) \coloneqq \bE_{(x,y)\sim\mu}  \left[ \loss(x, y, \btheta)\right],
\end{equation}
and for any input-target pair $(x,y)\in \bR^n\times\bR^k$, the loss function is defined as
\begin{equation}
\label{eq:loss}
\loss(x, y, \btheta) \coloneqq \Phi(\u^{\btheta, x}_T, y) + \int_0^T R(\theta_t) \dt
\end{equation}
Note that the regularizer $R$ could in general also depend on the state $u$ and $\Phi$ is the actual loss function upon which we want to act. It plays the same role as the loss function $\cL$ of section \ref{sec:stat-learning}. We can relate them by introducing the mapping $g:\bR^n \to \bR^k$ from section \ref{sec:cont-form}, and setting
$$
\Phi(\u^{\btheta, x}_T, y) = \cL( g(\u^{\btheta, x}_T), y ).
$$

Like in the setting of section \ref{sec:stat-learning}, we are only given a set of $N$ samples $\cS_N= \{(x_i, y_i)\}_{i=1}^N$ and the probability distribution $\mu$ is not exactly known. For a given control $\btheta\in\st$, each sample follows the dynamics
\begin{equation}
\label{eq:ODEsample}
\begin{cases}
\dot \u^{\btheta,i}_t &= f(u^{\btheta, i}_t, \theta_t),\quad \forall t\in (0,T) \\
 u_0^{\btheta, i} &= x_i
\end{cases}
\end{equation}
for $i=1,\dots,N$. We then perform empirical risk minimization taking, e.g., the uniform distribution $\mu_N = \frac 1 N \sum_{i=1}^N \delta_{(x_i, y_i)} $. The optimal control problem becomes
\begin{align}
\label{eq:ocs}
J^*_{\cS_N} = \inf_{\btheta\in\st} \cJ_{\cS_N}(\btheta),\quad \text{subject to \eqref{eq:ODEsample}},
\end{align}
where

\begin{align}
  \label{XXXXXX} % Why does the equation number not showing ???
  \cJ_{\cS_N}(\btheta) \coloneqq \bE_{(x,y) \sim \mu_N} \loss(x, y, \btheta) = \frac 1 N \sum_{i=1}^N  \Phi(u_T^{\btheta,i}, y_{i}) + \int_0^T R(\theta_t) \dt  
\end{align}

Note that the solutions of the sampled optimal control problem \eqref{eq:ocs} depend on the sample set $\cS_N$ and are therefore random variables.
However, for $\cS_N$ fixed,  problem \eqref{eq:ocs} is deterministic and can thus be solved with deterministic optimal control techniques.
In what follows, we adopt this viewpoint first to solve problem \eqref{eq:ocs} at any target accuracy. By this we mean the following: for any given target accuracy $\e>0$, we prove in Theorem \ref{th:conv-AMSA} that it is possible to numerically compute a control $\btheta_\e$ such that
$$
\cJ_{\cS_N}(\btheta_\e) - J^*_{\cS_N} \leq \e.
$$
We then come back to the probabilistic point of view and prove in Corollary \ref{cor:final-conv-result} that, provided that the number $N$ of samples is sufficiently large, there exists a constant $C>0$ such that
$$
\cJ_{\cS_N}(\btheta_\e) - J^* \leq C \e.
$$
with high probabillity. The derivation of the latter bound relies on the results of \cite{EHL2019}.

\section{Pontryagin's Maximum Principle and Method of Successive Approximations}
\label{sec:PMP-EMSA}

In this section, we focus on the \emph{sampled} optimal control problem \eqref{eq:ocs}. We first recall the necessary optimality conditions, usually known as the Pontryangin's Maximum Principle. We next introduce the main algorithm to solve the PMP which will be the starting point for our subsequent developments. In the following, the Euclidean norm of any vector $a\in \bR^n$ is denoted by $\Vert a \Vert$ and the scalar product with any other vector $b\in \bR^n$ is $a\cdot b$.

\subsection{Pontryangin's Maximum Principle}
We define the Hamiltoninan $H:[0,T]\times \bR^n\times \bR^n \times \Theta \to \bR$ as
\begin{equation}
\label{eq:H}
H(t,u,p,\theta) \coloneqq p \cdot f( u, \theta) - R(\theta).
\end{equation}
The following classical Pontryagin's Maximum Principle (PMP) gives necessary optimality conditions to problem \eqref{eq:ocs}.

\begin{theorem}
  Let $\btheta^*\in \st$ be an optimal control to \eqref{eq:ocs}.
  For $i=1,\dots,N$, let $\bu^{\btheta^*, i}$ be the associated state process with initial condition $x_i$.
  There exists a co-state process $\bp^{\btheta^*, i}\in \sp$ such that
\begin{align}
\dot u_t^{\btheta^*, i} &= f(  u_t^{\btheta^*, i} , \theta^*_t  ),\quad u_0^{\btheta^*, i} = x_i, \\
\dot p_t^{\btheta^*, i} &= -\nabla_u H( t,  u_t^{\btheta^*, i} , p_t^{\btheta^*, i}, \theta^*_t  ),\quad p_T^{\btheta^*, i} = -\nabla_u \Phi(u_T^{\btheta^*, i}, y_i),
\end{align}
and, for each $t\in [0,T]$,
\begin{align}
\label{eq:max-PMP-classic}
\frac 1 N \sum_{i=1}^N H( t,  u_t^{\btheta^*, i} , p_t^{\btheta^*, i}, \theta^*_t  )
\geq
\frac 1 N \sum_{i=1}^N H( t,  u_t^{\btheta^*, i} , p_t^{\btheta^*, i}, \theta  ),\quad \forall \theta \in \Theta.
\end{align}
\end{theorem}
The proof of this theorem and its variants can be found in any optimal control theory reference (see, e.g., \cite{Bertsekas1995, Clarke2005, AF2013}). We omitted the case of an abnormal multiplier, which will not be considered in the following.

We emphasize that the PMP is only a necessary condition, so there can be cases where the solutions to the PMP are not global optima for \eqref{eq:ocs}. Nevertheless, in practice the PMP often gives good solution candidates, and when certain convexity assumptions are satisfied the PMP becomes sufficient (see \cite{BP2007}). In the next section, we discuss the numerical methods that we take as a starting point to solve the PMP.

\subsection{MSA and Extended-MSA}
A classical algorithm to find $\btheta^*$ and the corresponding forward and co-state dynamics $\bu^{\btheta^*, i}$ and $\bp^{\btheta^*, i}$ is the Method of Successive Approximations (see \cite{CL1982}). It is a fixed-point method based on the following steps. Starting from an initial guess of the optimal control $\btheta^0$, for each $k\geq 0$
%\olga{Clash of notations with $\bR^k$ and iteration index $k$. Change $\bR^k$ to $\bR^q$? Or what if $n \to n_X$ and $k \to n_Y$?}
we first solve for $i=1,\dots, N$ the forward dynamics
\begin{equation}
\label{eq:fwd-exact}
\dot u_t^{\btheta^k,i}
= \nabla_p H( t, u_t^{\btheta^k, i} , p_t^{\btheta^k, i}, \theta^k_t  )
= f( u_t^{\btheta^k, i} , \theta^k_t  ),\quad u^{\btheta^k, i}_0 = x_i. 
\end{equation}
The dynamics $\bu^{\btheta^k,i} $ allows us to compute the backward dynamics
\begin{equation}
\label{eq:bwd-exact}
\dot p_t^{\btheta^k,i} = - \nabla_u H( t, u_t^{\btheta^k, i} , p_t^{\btheta^k, i}, \theta^k_t  ),\quad p^{\btheta^k, i}_T = - \nabla_u \Phi(u^{\btheta^k, i}_T, y_i). 
\end{equation}
Finally, we update the control by using the maximization condition \eqref{eq:max-PMP-classic},
\begin{equation}
\label{eq:max-step-PMP}
\theta^{k+1}_t \in  \argmax_{\theta\in \Theta} \frac 1 N \sum_{i=1}^N H( t,  u_t^{\btheta^k, i} , p_t^{\btheta^k, i}, \theta  ),\quad \forall t \in [0,T].
\end{equation}

In this form, \MSA converges only locally when it is initialized with a starting control guess that is sufficiently close to an optimal control $\btheta^*$. To overcome this limitation, an Extended MSA algorithm (\EMSA) based on an augmented Lagrangian strategy has been introduced in \cite{LCTE2018}. This algorithm is the starting point for our subsequent developments. It works as follows: fix some $\rho>0$ and define the augmented Halmitonian
\begin{equation}
\label{eq:augH}
\wH (t, u, p, \theta, v, q) \coloneqq H(t, u, p, \theta) -\frac \rho 2 \Vert v - f( u, \theta) \Vert^2 - \frac \rho 2 \Vert q + \nabla_u H(t, u, p, \theta) \Vert^2.
\end{equation}
We can now formulate an extended PMP based on $\wH$.

\begin{proposition}[Extended PMP, see \cite{LCTE2018}]
Let $\btheta^*\in \st$ be an optimal control to \eqref{eq:ocs}. For $i=1,\dots,N$, let $\bu^{\btheta^*, i}$ be the associated state process with initial condition $x_i$. There exists a co-state process $\bp^{\btheta^*, i}\in \sp$ such that
\begin{align}
\dot u_t^{\btheta^*, i} &= \nabla_p \wH( t,  u_t^{\btheta^*, i} , p_t^{\btheta^*, i}, \theta^*_t  ),\quad u_0^{\btheta^*, i} = x_i, \\
\dot p_t^{\btheta^*, i} &= -\nabla_u \wH( t,  u_t^{\btheta^*, i} , p_t^{\btheta^*, i}, \theta^*_t  ),\quad p_T^{\btheta^*, i} = -\nabla_u \Phi(u_T^{\btheta^*, i}, y_i),
\end{align}
and, for each $t\in [0,T]$,
\begin{equation}
\label{eq:max-wH}
\frac 1 N \sum_{i=1}^N \wH( t,  u_t^{\btheta^*, i} , p_t^{\btheta^*, i}, \theta^*_t  )
\geq
\frac 1 N \sum_{i=1}^N \wH( t,  u_t^{\btheta^*, i} , p_t^{\btheta^*, i}, \theta  ),\quad \forall \theta \in \Theta.
\end{equation}
\end{proposition}
The \EMSA algorithm consists in applying the MSA algorithm with the augmented Hamiltonian $\wH$ instead of with $H$. Since $\nabla_p \wH = \nabla_p H = f$ and $\nabla_u \wH = \nabla_u H$, steps \eqref{eq:fwd-exact} and \eqref{eq:bwd-exact} remain the same as in MSA and the maximation step \eqref{eq:max-step-PMP} is replaced by
\begin{equation}
\label{eq:max-step}
\theta^{k+1}_t \in  \argmax_{\theta\in \Theta}  \frac 1 N \sum_{i=1}^N \wH(t, u_t^{\btheta^k,i}, p_t^{\btheta^k, i} , \theta),\quad \forall t \in [0,T].
\end{equation}
It has been proven in \cite{LCTE2018} that if the parameter $\rho$ is taken sufficiently large,  the \EMSA algorithm converges to the set of solution of the extended PMP for any initial guess of the control $\btheta^0$. However, this scheme cannot be realized in practice without discretizing. This introduces a perturbation of the time-continuous formulation of the algorithm since the resulting trajectories will be approximations of the continuous one. To guarantee convergence to the exact continuous solution, it is necessary to identify suitable approximation error tolerances that still guarantee convergence to the exact solution. This motivates to introduce an Adaptive MSA scheme that we describe in the next section. We prove that if the forward and backward propagations are solved with increasing accuracy at each step, then the algorithm converges towards a \emph{continuous} solution of the Extended PMP.

\section{\AMSA: Adaptive Method of Successive Approximations}
\label{sec:EMSA}
The algorithm requires defining forward and backward time-integration schemes, which we introduce in section \ref{sec:practical-realization}. We next present the algorithm and how predictions are made in sections \ref{sec:AMSA-learning} and \ref{sec:AMSA-predictions}. We end up by proving convergence in section \ref{sec:conv-AMSA}.

\subsection{Routines \fwd and \bwd}
\label{sec:practical-realization}
In the following, we work with time-integration routines \fwd and \bwd which can be cast in the abstract framework that we next describe. Relevant particular instances of it will be continuous Galerkin (cG) or Runge-Kutta collocation (RK-C) methods.

\paragraph*{Setting:} We describe the framework in the case of the forward dynamics
\begin{equation}
\begin{cases}
\dot u_t &= f(u_t), \quad \forall t \in (0,T)\\
u_0 &= x 
\end{cases}
\label{eq:fwd-dyn}
\end{equation}
for a generic Lipschitz continuous function $f:\bR^n\to \bR^n$. The backward dynamics can be deduced similarly.

Let $0=t_0<t_1<\dots<t_L=T$ be a partition of $[0,T]$ and let $\cT_\ell\coloneqq (t_{\ell-1}, t_\ell]$, and $h_\ell\coloneqq t_\ell - t_{\ell-1}$. For $q\in \bN$, let $\cV_q^\cT$ be the space of continuous functions that are piecewise polynomials of degree $q$ in the time mesh $\cT = \cup_{\ell=1}^L \cT_\ell$, i.e., 
$$
\cV_q^\cT \coloneqq \{ v \in  \cC^0([0,T]; \bR^n) \, : \,  v\vert_{\cT_\ell} \in \bP^{(q)} (\cT_\ell),\quad \ell =1,\dots, L \},
$$
where
$$
\bP^{(q)}(\cT_\ell) \coloneqq \{ v \in  \cC^0(\cT_\ell; \bR^n) \, : \, \forall \, t\in \cT_\ell,\, v(t) = \sum_{j=0}^q t^j v_j,\, v_j \in \bR^n  \}
$$
is the space of polynomials of degree $q$ in the interval $\cT_\ell$. 

Introducing a \emph{projection operator}
\begin{equation}
\Pi^{(q)}: \cC^0([0,T]; \bR^n) \mapsto \cV_q^\cT
\end{equation}
the time discrete approximation $U$ to the solution $u$ is defined as follows: we seek $U\in \cV_q^\cT$ satisfying the initial condition $U(0)=x$ as well as
\begin{equation}
\label{eq:Z}
U_t' = \Pi^{(q-1)} f( U_t),\quad \forall t \in \cT_\ell,
\end{equation}
for $\ell=1,\dots,L$. Note that since both terms belong to $\bP^{(q-1)}(\cT_\ell)$, \eqref{eq:Z} admits a Galerkin formulation
\begin{equation}
\label{eq:Zgalerkin}
\int_{\cT_\ell} v \cdot U_t' \,\dt = \int_{\cT_\ell} v \cdot \Pi^{(q-1)} f( U_t)\, \dt, \quad \forall v \in \bP^{(q)}(\cT_\ell).
\end{equation}
In the following, we use mainly \eqref{eq:Z} but \eqref{eq:Zgalerkin} is of interest since it connects (RK-C) methods with (cG) methods. It is proven in \cite{AMN2009}, that the continuous Galerkin method corresponds to the choice $\Pi^{(q-1)} \coloneqq P^{(q-1)}$, with $P^{(q-1)}$ denoting the (local) $L^2$ orthogonal projection onto $\bP^{(q-1)}(\cT_\ell)$ for each $\ell$. Runge-Kutta collocation methods with pairwise distinct nodes in $\cT_\ell$ can be obtained by choosing $\Pi^{(q-1)}\coloneqq I^{(q-1)}$ with $I^{(q-1)}$ denoting the interpolation operator by elements of $\bP^{(q)}(\cT_\ell) $ at the nodes $t_{\ell-1}+\alpha_i (t_\ell-t_{\ell-1})$, $i=1,\dots, q$, $\ell=1,\dots, L$, with appropriate weights $0\leq \alpha_1 < \dots < \alpha_q \leq 1$. All RK-C methods with pairwise distinct nodes in $[0,1]$ can be obtained by applying appropriate numerical quadrature to continuous Galerkin methods.

\paragraph{$L^2$ error estimation:}
%The implementation of \AMSA requires a computable estimate of the $L^2$ error $\Vert z - Z \Vert_{L^2([0,T], \bR^n)}$. There are numerous contributions on the topic of rigorously estimating the approximation error of ODEs. \om{Give some refs}. 
One simple way of estimating the error between $u$ and $U$ is via the following Grönwall inequality. Denoting $e_t = u_t - U_t$, we have
$$
\dot e_t = f( u_t) - \Pi^{(q-1)} f( U_t) = \left( f( u_t) - f( U_t) \right) + \left( f( U_t) - \Pi^{(q-1)} f( U_t) \right)
$$
Multiplying by $e_t$, we get
\begin{align}
\frac 1 2 \frac{\rm d}{\dt} e_t^2
&= e_t \left( f( u_t) - f( U_t) \right) + e_t \left( f( U_t) - \Pi^{(q-1)} f( U_t) \right) \\
&\leq L e_t^2 + \vert e_t \vert \vert f( U_t) - \Pi^{(q-1)} f( U_t) \vert \\
&\leq (L+1/2) e_t^2 + \frac 1 2 \vert f( U_t) - \Pi^{(q-1)} f( U_t) \vert^2,
\end{align}
where $L$ is the Lipschitz constant on the second variable of $g$. By the Grönwall inequality, since $e(0)=0$,
$$
e_t^2 \leq \int_0^t  \vert f(U_s) - \Pi^{(q-1)} f(U_s) \vert^2 e^{(2L+1)(t-s)} \ds
$$
which yields an $L^2$ estimate of the error
\begin{align}
\Vert u - U \Vert^2_{L^2([0,T], \bR^n)} 
&\leq  \int_0^T \int_0^t  \vert f( U_s) - \Pi^{(q)} f( U_s) \vert^2 e^{(2L+1)(t-s)} \ds \dt  \\
&\leq \frac{e^{(2^L+1)T}}{2L+1}
\Vert f - \Pi^{(q-1)} f \Vert^2_{L^2([0,T], \bR^n)}
\label{eq:err-bound-prop}
\end{align}
In the following, we denote $\Pi_\zeta$ a projector delivering an accuracy $\eta \geq 0$ for the involved dynamics
\begin{equation}
\Vert f - \Pi^{(q-1)} f \Vert_{L^2([0,T], \bR^n)}
\leq \eta .
\label{eq:err-pi-f}
\end{equation}
From inequality \eqref{eq:err-bound-prop}, such a projector gives an accuracy in the solution $U$ which is bounded by
$$
\Vert u - U \Vert_{L^2([0,T], \bR^n)} \leq \frac{e^{(2^L+1)T/2}}{(2L+1)^{1/2}} \eta
$$

\paragraph*{The routines:} In the following, the routine
$$
\bu^\eta = \fwd( \eta; f, x) 
$$
yields an evolution
$$
\bu^\eta = \{  u^\eta_t \, :\, t \in [0,T]\}
$$
that approximates the exact solution $\bu$ of the \emph{forward} dynamics \eqref{eq:fwd-dyn} with a projector $\Pi_\eta$. The routine
$$
\bp^\eta = \bwd( \eta; f, x) 
$$
works similarly for the \emph{backward} dynamics $\dot p_t = f( p_t)$ with $p_T = x$.

\subsection{The learning phase of the algorithm}
\label{sec:AMSA-learning}

The \AMSA algorithm consists in performing steps \eqref{eq:fwd-exact} and \eqref{eq:bwd-exact} of the original MSA but with a numerical time-integrators which give inexact trajectories. We carry the discussion assuming that we use the routines \fwd and \bwd introduced above. At every iteration $k$, both routines involves projectors $\Pi_k$ that deliver an accuracy $\e_k$ which is yet to be determined. The accuracy will be tightened as $k$ increases in a way that still guarantees convergence to the exact, continuous solution \eqref{eq:ocs} of the sampled PMP problem. We deduce its value later on from the convergence analysis of section \ref{sec:conv-AMSA}.

Starting from an initial guess of the optimal control $\btheta^0$, for each $k\geq 0$ we first solve at accuracy $\e_k$ the forward dynamics
\begin{equation}
\label{eq:inexact-fwd}
\bu^{\btheta^k, i, \e_k} = 
\fwd\left( \e_k;  (t, u) \mapsto f( u , \theta_t^k ), x_i \right), \quad \forall i=1,\dots, N.
\end{equation}
Satisfying this accuracy may require to adapt the time discretization mesh and/or the numerical scheme as discussed in the previous section.

We then use the final state $u_T^{\btheta^k, i,\e_k} $ to compute an approximation of the \emph{backward} dynamics at the same accuracy $\e_k$
\begin{equation}
\label{eq:inexact-bwd}
\bp^{\btheta^k, i, \e_k} = 
\bwd\left( \e_k;  (t, p) \mapsto - \nabla_u H( t, u_t^{\btheta^k, i, \e_k} , p, \theta^k_t  ), - \nabla_u \Phi(u^{\btheta^k, i, \e_k}_T, y_i) \right), \quad \forall i=1,\dots, N.
\end{equation}
This step may, again, require adaptivity to satisfy the target tolerance. To update the control, instead of finding the exact maximum like in \eqref{eq:max-step}, it is in fact sufficient to find controls that are away by a factor $0<\gamma_k \leq 1$ of the maximum and for which $\gamma_k \to 1$ as $k\to \infty$. In other words, it suffices to find for all $t\in [0,T]$, a control $\theta^{k+1}_t \in \Theta$ such that
\begin{equation}
\label{eq:inexact-max-step}
\frac 1 N \sum_{i=1}^N
\wH(t, u_t^{\btheta^k, i, \e_k}, p_t^{\btheta^k, i, \e_k} , \theta^{k+1}_t)
\geq
\gamma_k \max_{\theta \in \Theta} \frac 1 N \sum_{i=1}^N  \wH(t, u_t^{\btheta^k, i, \e_k}, p_t^{\btheta^k, i, \e_k} , \theta),\quad \forall t \in [0,T].
\end{equation}
Note that when $\gamma_k=1$, we fall back to the case where we look for the exact maximum.
Algorithm \ref{alg:amsa} summarizes the whole procedure by describing a routine $\AMSA[\cS, \tau]$ which computes a control dynamics $\hat \btheta$ that solves the extended PMP problem until the difference between two successive iterates is smaller than $\tau>0$. The algorithm also yields a numerical scheme associated to the parameter $\hat \e$ which is the  accuracy of the forward and backward solvers used in the last iteration. These elements are then used for prediction as explained in the next section.

The practical implementation of the algorithm requires to specify the tolerances $\e_k$. They can be theoretically derived from the convergence analysis of section \ref{sec:conv-AMSA}. However, since the bounds of the analysis may not be sharp and that it involves quantities which are difficult to estimate in practice, devote section \ref{sec:impl} to give indications on how to implement \AMSA in practice.

\begin{algorithm}
   \caption{Learning Algorithm: $\AMSA[\cS_N, \tau] \rightarrow [\hat \btheta, \hat \e] $}
  \label{alg:amsa}
  \begin{algorithmic}[1]
   \STATE \textbf{Input:} $\cS_N = \{ x_i, y_i \}_{i=1}^N$, target tolerance $\tau$
   \STATE \textbf{Set of internal parameters:} $\btheta^0$, $\rho$, maximum number of iterations $k_{\max}$.
   \STATE $k \gets 0$
   \STATE $J, J_{\text{old}}\gets J(\btheta^0)$
   \STATE $\Delta \gets \tau +1$
\WHILE{ $\Delta > \tau$ or $k\leq \kmax$}
	\STATE Update $\e_k$ and $\gamma_k$
	\FOR{$i$ in $\{1, \dots, N\}$} \COMMENT{parallel}
		\STATE $\bu^{\btheta^k, i, \e_k} = \fwd\left( \e_k;  (t, u) \mapsto f( u , \theta^k_t ), x_i \right)$
		\STATE $\bp^{\btheta^k, i, \e_k} = \bwd\left( \e_k;  (t, p) \mapsto - \nabla_x H( t, u_t^{\btheta^k, i, \e_k} , p, \theta^k_t  ), - \nabla_u \Phi(u^{\btheta^k, i, \e_k}_T, y_i) \right)$
	\ENDFOR
	\FORALL{$t\in [0,T]$} \COMMENT{parallel}
		\STATE Find $\theta^{k+1}_t$ s.t. $$\frac 1 N \sum_{i=1}^N\wH(t, u_t^{\btheta^k,i,\e_k}, p_t^{\btheta^k, i, \e_k} , \theta^{k+1}_t)
\geq \gamma_k \max_{\theta \in \Theta}  \frac 1 N \sum_{i=1}^N\wH(t, u_t^{\btheta^k,i,\e_k}, p_t^{\btheta^k, i,\e_k} , \theta)$$
	\ENDFOR
	\STATE $J_{\text{old}} \gets J$; $\quad J\gets J(\btheta^{k+1})$; $\quad\Delta \gets J_{\text{old}}  - J$
\ENDWHILE
\STATE$\hat \btheta \gets \btheta^{k+1}; \quad \hat \e\gets \e_k$
\STATE \textbf{Output:} $[\hat \btheta, \hat \e]$
\end{algorithmic}
\end{algorithm}

\subsection{Predictions}
\label{sec:AMSA-predictions}
The routine $\AMSA[\cS, \tau]$ gives a discrete control $\hat \btheta$ and an accuracy $\hat \e$. We also have a certain numerical scheme for the forward propagator, which can be interpreted as the final network architecture. As a result, for a given input data $x$, we can predict the output by first computing
$$
\hat \bu = \fwd\left( \hat \e;  (t, u) \mapsto f( u , \hat \btheta ), x \right),
$$
and then taking
$$
\hat y \coloneqq g(\hat u_T)
$$
as the approximation of the true $y$. In other words, our neural network is the mapping
\begin{align}
\cN \cN : \bR^n &\to \bR^k   \\
x &\to \cN\cN (x) \coloneqq g(\hat u_T) = g\left( \fwd\left( \hat \e;  (t, u) \mapsto f( u , \hat \btheta ), x \right)(t=T) \right)
\end{align}
We summarize the prediction pipeline in Algorithm \ref{alg:NN}.

\begin{algorithm}
   \caption{Prediction algorithm: $\cN\cN(x) \rightarrow \hat y$}
  \label{alg:adnn}
  \begin{algorithmic}[1]
   \STATE \textbf{Input:} Observation $x\in \bR^n$, Neural network architecture $[\hat \btheta, \hat \e]$ (from $\AMSA[\cS_N, \tau]$).
   \STATE $\hat \bu \gets \fwd\left( \hat \e;  (t, u) \mapsto f( u , \hat \btheta ), x \right)$
   \STATE \textbf{Output:} $g(\hat u_T)$
\end{algorithmic}
\label{alg:NN}
\end{algorithm}

\section{A priori convergence of \AMSA}
\label{sec:conv-AMSA}
In this section, we prove that, for suitably chosen tolerances $(\e_k)_k$, the sequence of controls $(\btheta^k)_{k\geq 0}$ computed with \AMSA is a minimizing sequence for problem \eqref{eq:ocs}, that is,
$$
\cJ_{\cS_N}(\btheta^k) \xrightarrow{k\rightarrow\infty} J^*_{\cS_N}.
$$
Our analysis is built upon the one presented in \cite{LCTE2018} for the exact \EMSA  to which it is necessary to add perturbative arguments due to the inexact propagations.

We work with the same continuity assumptions as in \cite{LCTE2018} for the ideal case of exact propagations:
\begin{itemize}
\item[(A1)] $\Phi$ is twice continuously differentiable. $\Phi$ and $\nabla \Phi$ satisfy the following Lipschitz condition: there exists $K>0$ such that
$$
\vert \Phi(u) - \Phi(u') \vert + \Vert \nabla \Phi (u) - \nabla\Phi(u') \Vert \leq K \Vert u - u' \Vert, \quad \forall (u, u') \in \bR^n.
$$
\item[(A2)] For all $t\in [0,T]$ and $\theta \in \Theta$, $u\to f( u, \theta )$ is twice continuously differentiable and satisfying the following Lipschitz condition: there exists $K'>0$ such that
$$
\Vert f( u, \theta) - f( u', \theta) \Vert + \Vert \nabla_u f( u, \theta) - \nabla_u f( u', \theta) \Vert_2 \leq K \Vert u - u' \Vert , \quad \forall (u, u') \in \bR^n,
$$
where $\Vert \cdot \Vert_2$ denotes the induced 2-norm.
\end{itemize}

\paragraph*{Intermediate problem:} We first study the following intermediate problem. Consider a fixed input-output pair $(x, y)$.
For a given control $\btheta\in \st$, let $\bu^{\btheta, x}$ and $\bp^{\btheta, x}$ be the exact solutions of the forward and backward propagations using $x$ and $y$ in the initial and final conditions,
\begin{align}
\begin{cases}
\dot u_t^{\btheta, x} &= f(  u_t^{\btheta, x} , \theta_t  ),\quad u_0^{\btheta, x} = x, \\
\dot p_t^{\btheta, x} &= -\nabla_u H( t,  u_t^{\btheta, x} , p_t^{\btheta, x}, \theta_t  ),\quad p_T^{\btheta, x} = -\nabla_u \Phi(u_T^{\btheta, x}, y),
\end{cases}
\end{align}
 Let $\bu^{\btheta, x, \zeta}$ and $\bp^{\btheta, x, \zeta}$ be the outputs of the same propagation but with accuracy $\zeta>0$,
\begin{align}
\begin{cases}
\bu^{\btheta, x, \zeta} &= \fwd\left( \zeta;  (t, u) \mapsto f( u , \btheta ), x \right) \\
\bp^{\btheta, x, \zeta} &= 
\bwd\left( \zeta;  (t, p) \mapsto - \nabla_u H( t, u_t^{\btheta, x, \zeta} , p, \theta_t  ), - \nabla_u \Phi(u^{\btheta, x, \zeta}_T, y) \right).
\end{cases}
\label{eq:approx-fwd-bwd}
\end{align}
 Similarly, $\bu^{\bphi, x, \eta}$ and $\bp^{\bphi, x, \eta}$ denote the $\eta$-accurate solutions but for another control $\bphi\in \st$. 
 
The loss function associated to $\bu^{\btheta, x, \zeta}$ is denoted
\begin{equation}
\label{eq:loss2}
\loss(x, y, \btheta, \zeta) \coloneqq \Phi(\u^{\btheta, x, \zeta}_T, y) + \int_0^T R(\theta_t) \dt,
\end{equation}
and similarly for $\bu^{\bphi, x, \eta}$.
In addition, we define the quantity
\begin{equation}
\label{eq:deltaH}
\Delta H^{x, \zeta}_{\bphi, \btheta}(t) \coloneqq H(t, u_t^{\btheta, \zeta}, p_t^{\btheta, \zeta}, \varphi_t)-H(t, u_t^{\btheta, \zeta}, p_t^{\btheta, \zeta}, \theta_t),
\end{equation}
which will play an essential role in what follows.

Lemma \ref{lem:diffLoss} is an important building-block in the proof of convergence of \AMSA. To not interrupt the flow of reading, we have deferred its proof to Appendix \ref{app:proof-prop}.

\begin{lemma}
\label{lem:diffLoss}
Let $\btheta,\, \bphi\in \st$ be two controls and let $(x, y)\in \bR^n\times \bR^k$ be an input-output pair. Let $(\bu^{\btheta, x, \zeta},\bp^{\btheta, x, \zeta})$ and $(\bu^{\bphi, x, \eta},\bp^{\bphi, x, \eta})$ be solutions of the forward and backward associated problems with accuracy $\zeta$ and $\eta$ like in \eqref{eq:approx-fwd-bwd}. Then, there exist a constant $C>0$ independent of $\btheta,\, \bphi$ but  dependent on $T$, $\Vert \Pi_\zeta\Vert$ and $\Vert \Pi_\eta\Vert$ such that, if $\eta\leq 1$,
\begin{align}
\label{eq:diffLoss}
&\loss(x, y, \bphi, \eta) -  \loss(x, y, \btheta, \zeta)  \\
&\qquad \leq - \int_0^T \Delta H^{x,\zeta}_{\bphi, \btheta}(t)\dt  + C \left( (\eta + \zeta)^2  + \Vert \nabla_w (\Delta H^{x,\zeta}_{\bphi, \btheta} )\Vert_{L^2([0,T])}^2 \right)
\end{align}
\end{lemma}

\paragraph*{Convergence of \AMSA:}
We can use Lemma \ref{lem:diffLoss} to prove our main convergence result which we give in the following theorem. It involves the quantity
\begin{align}
\lambda_k^2 &\coloneqq \frac 1 N \sum_{i=1}^N \Vert \nabla_w (\Delta H^{i, \e_k}_{\btheta^{k+1}, \btheta^k} )\Vert_{L^2([0,T])}^2 \\
&=\frac 1 N \sum_{i=1}^N \left( \int_0^T \Vert f( u_t^{\btheta^k, i, \e_k}, \theta_t^{k+1}) - f( u_t^{\btheta^k, i, \e_k}, \theta_t^k) \Vert^2 \dt \right. \\
&+ \left. \int_0^T \Vert \nabla_u H(t, u_t^{\btheta^k, i, \e_k}, p_t^{\btheta^k, i, \e_k}, \theta_t^{k+1}) - \nabla_u H(t, u_t^{\btheta^k, i, \e_k}, p_t^{\btheta^k, i, \e_k}, \theta_t^k) \Vert^2 \dt \right)
\end{align}

\begin{theorem}
\label{th:conv-AMSA}
Let $\delta > 0$ be a fixed constant and suppose that we run the \AMSA algorithm with the following accuracies:
\begin{itemize}
\item $(\e_k)$ is a positive and decreasing sequence such that
$$
0\leq \e_k \leq  \min\left( \e_{k-1}, \frac{ \delta \lambda_k^2 }{ 4C + \rho N^{-1/2}\lambda_k}\right),\quad \forall k \geq 1.
$$
\item The maximization step \eqref{eq:inexact-max-step} is performed with the parameter
$$
\gamma_k \geq
\frac{\big\vert \sum_{i=1}^N
\wH(t, u_t^{\btheta^k, i, \e_k}, p_t^{\btheta^k, i, \e_k} , \theta^{k+1}_t)/N \big\vert}{\delta \lambda_k^2 + \big\vert \sum_{i=1}^N
\wH(t, u_t^{\btheta^k, i, \e_k}, p_t^{\btheta^k, i, \e_k} , \theta^{k+1}_t)/N \big\vert}
,\quad \forall k \in \bN.
$$
\end{itemize}
Then, if $\rho > 2(C+2\delta+\e_0)$, where $C$ is the constant of Lemma \ref{lem:diffLoss}, the sequence of controls $(\btheta^k)$ of \AMSA is a minimizing sequence for problem \eqref{eq:ocs}, that is,
$$
\cJ_{\cS_N}(\btheta^k) \to_{k\to \infty} J^*_{\cS_N}.
$$
Furthermore, $\lambda^2_k\to 0$ as $k\to \infty$ and we have
$$
\cJ_{\cS_N}(\btheta^{k+1}) - \cJ_{\cS_N}(\btheta^k) \leq - \kappa \lambda^2_k, \quad \kappa \coloneqq \vert C+2\delta - \frac \rho 2 \vert >0,\quad \forall k\in \bN.
$$
\end{theorem}

\begin{proof}
We start by applying inequality \eqref{eq:diffLoss} from Lemma \ref{lem:diffLoss} to the controls $\bphi = \btheta^{k+1}$ and $\btheta=\btheta^k$ of the \AMSA algorithm, with accuracies $\eta = \e_{k+1}$ and $\zeta = \e_k$ and for  the  samples $(x,y)=(x_i, y_i)$. We then take the average over $i$ and obtain
\begin{align}
\cJ_{\cS_N}(\btheta^{k+1}) - \cJ_{\cS_N}(\btheta^k)&= \frac 1 N \sum_{i=1}^N \loss(x_i, y_i, \btheta^{k+1}, \e_{k+1}) -  \loss(x_i, y_i, \btheta^k, \e_k)  \\
&\leq -\, \frac 1 N \sum_{i=1}^N \int_0^T \Delta H^{i, \e_k}_{\btheta^{k+1}, \btheta^k}(t)\dt +C \left( (\e_k + \e_{k+1})^2  + \lambda_k^2\right) \label{eq:diffJk}
\end{align}
%Furthermore, since
%\begin{align}
%\Delta H^i_{\btheta^{k+1}, \btheta^k}(t)
%= H(t, u_t^{\btheta^k, i, \zeta},  p_t^{\btheta^k, i, \zeta}, \theta^{k+1}_t)
%- H(t, u_t^{\btheta^k, i, \zeta},  p_t^{\btheta^k, i, \zeta}, \theta^k_t),
%\end{align}
We first do the proof when $\gamma_k=1, \forall k\geq 1$. In this case, from the maximization step \eqref{eq:inexact-max-step} in \AMSA,
\begin{align}
\frac 1 N \sum_{i=1}^N
\wH(t, u_t^{\btheta^k, i, \e_k}, p_t^{\btheta^k, i, \e_k} , \theta^{k+1}_t, \dot u_t^{\btheta^k, i, \e_k}, \dot p_t^{\btheta^k, i, \e_k})
&\geq
\frac 1 N \sum_{i=1}^N
\wH(t, u_t^{\btheta^k, i, \e_k}, p_t^{\btheta^k, i, \e_k} , \theta^{k}_t, \dot u_t^{\btheta^k, i, \e_k}, \dot p_t^{\btheta^k, i, \e_k}).
\end{align}
Recalling definition \eqref{eq:augH} for $\wH$, we reassemble the terms in the above inequality to derive
\begin{align}
&-\, \frac 1 N \sum_{i=1}^N \int_0^T \Delta H^{i, \e_k}_{\btheta^{k+1}, \btheta^k}(t)\dt \\
& \leq
- \frac{\rho}{2N} \sum_{i=1}^N \int_0^T
\left(
\Vert \dot u_t^{\btheta^k, i, \e_k} - f( u_t^{\btheta^k, i, \e_k}, \theta^{k+1}_t) \Vert^2
+
\Vert \dot p_t^{\btheta^k, i, \e_k} - \nabla_u H(t, u_t^{\btheta^k, i, \e_k}, p_t^{\btheta^k, i, \e_k},  \theta^{k+1}_t) \vert^2 
\right)\dt  \\
&+
\frac{\rho}{2N} \sum_{i=1}^N \int_0^T
\left(
\Vert \dot u_t^{\btheta^k, i, \e_k} - f( u_t^{\btheta^k, i, \e_k}, \theta^{k}_t) \Vert^2
+
\Vert \dot p_t^{\btheta^k, i, \e_k} - \nabla_u H(t, u_t^{\btheta^k, i, \e_k}, p_t^{\btheta^k, i, \e_k},  \theta^{k}_t) \Vert^2 
\right)\dt
\label{eq:bound-Delta}
\end{align}
Since by construction
$$
\dot u_t^{\btheta^k, i, \e_k}  = \Pi_{\e_k} f( u_t^{\btheta^k, i, \e_k}, \theta^{k}_t)
$$
and
$$
\Vert  f( u_t^{\btheta^k, i, \e_k}, \theta^{k}_t) - \Pi_{\e_k} f( u_t^{\btheta^k, i, \e_k}, \theta^{k}_t) \Vert_{L^2([0,T], \bR^n)} \leq \e_k,
$$
we have by Cauchy-Schwartz inequality
\begin{align}
\int_0^T
&\Vert \dot u_t^{\btheta^k, i, \e_k} - f( u_t^{\btheta^k, i, \e_k}, \theta^{k+1}_t) \Vert^2 \dt \\
&= \int_0^T \Vert  f( u_t^{\btheta^k, i, \e_k}, \theta^{k}_t) - \Pi f( u_t^{\btheta^k, i, \e_k}, \theta^{k}_t) \Vert^2  +\Vert f( u_t^{\btheta^k, i, \e_k}, \theta^{k}_t) - f( u_t^{\btheta^k, i, \e_k}, \theta^{k+1}_t) \Vert^2  \\
& - 2 \left<  f( u_t^{\btheta^k, i, \e_k}, \theta^{k}_t) - \Pi f( u_t^{\btheta^k, i, \e_k}, \theta^{k}_t), f( u_t^{\btheta^k, i, \e_k}, \theta^{k}_t) - f( u_t^{\btheta^k, i, \e_k}, \theta^{k+1}_t) \right> \dt \\
&\geq \int_0^T \Vert  f( u_t^{\btheta^k, i, \e_k}, \theta^{k}_t) - \Pi f( u_t^{\btheta^k, i, \e_k}, \theta^{k}_t) \Vert^2  +\Vert f( u_t^{\btheta^k, i, \e_k}, \theta^{k}_t) - f( u_t^{\btheta^k, i, \e_k}, \theta^{k+1}_t) \Vert^2 \,\dt \\
& - 2\e_k
\left( \int_0^T \Vert f( u_t^{\btheta^k, i, \e_k}, \theta^{k}_t) - f( u_t^{\btheta^k, i, \e_k}, \theta^{k+1}_t) \Vert^2 \dt \right)^{1/2}
\end{align}
%and by Cauchy-Schwartz inequality
%\begin{align}
%&\int_0^T \left<  f( u_t^{\btheta^k, i, \e_k}, \theta^{k}_t) - \Pi f( u_t^{\btheta^k, i, \e_k}, \theta^{k}_t), f( u_t^{\btheta^k, i, \e_k}, \theta^{k}_t) - f( u_t^{\btheta^k, i, \e_k}, \theta^{k+1}_t) \right> \dt \\
%&\quad\leq \e_k \left( \int_0^T \Vert f( u_t^{\btheta^k, i, \e_k}, \theta^{k}_t) - f( u_t^{\btheta^k, i, \e_k}, \theta^{k+1}_t) \Vert^2 \dt \right)^{1/2}
%\end{align}
%we get
%\begin{align}
%&\int_0^T
%\Vert \dot u_t^{\btheta^k, i, \e_k} - f( u_t^{\btheta^k, i, \e_k}, \theta^{k+1}_t) \Vert^2 \dt \\
%&\quad\geq
%\int_0^T
%\Vert  f( u_t^{\btheta^k, i, \e_k}, \theta^{k}_t) - \Pi f( u_t^{\btheta^k, i, \e_k}, \theta^{k}_t) \Vert^2
%+\Vert f( u_t^{\btheta^k, i, \e_k}, \theta^{k}_t) - f( u_t^{\btheta^k, i, \e_k}, \theta^{k+1}_t) \Vert^2  \\
%& \quad- 2\e_k
%\left( \Vert f( u_t^{\btheta^k, i, \e_k}, \theta^{k}_t) - f( u_t^{\btheta^k, i, \e_k}, \theta^{k+1}_t) \Vert^2 \dt \right)^{1/2}
%\end{align}
Since a similar bound holds for $\int_0^T \Vert \dot p_t^{\btheta^k, i, \e_k} - \nabla_u H(t, u_t^{\btheta^k, i, \e_k}, p_t^{\btheta^k, i, \e_k},  \theta^{k+1}_t) \Vert^2 \dt$, we derive that
\begin{align}
&-\, \frac 1 N \sum_{i=1}^N \int_0^T \Delta H^{i, \e_k}_{\btheta^{k+1}, \btheta^k}(t)\dt
\leq  
-\frac{\rho}{2}\lambda_k^2 +  \rho N^{-1/2} \e_k \lambda_k
\label{eq:bound-Delta-gamma-1}
\end{align}
Using \eqref{eq:bound-Delta-gamma-1}, we can further bound \eqref{eq:diffJk} to obtain
\begin{align}
\cJ_{\cS_N}(\btheta^{k+1}) - \cJ_{\cS_N}(\btheta^k) 
&\leq
-\frac{\rho}{2}\lambda_k^2 +  \rho N^{-1/2} \e_k \lambda_k
+C \left( (\e_k + \e_{k+1})^2  + \lambda_k^2\right) \\
&\leq \left(C - \frac{\rho}{2}\right) \lambda_k^2 + \e_k (4C +  \rho N^{-1/2} \lambda_k)
\label{eq:diff-J}
\end{align}
where we have used that $(\e_k)$ is a decreasing sequence.

Since, for a fixed $\delta>0$, we have chosen $\e_k$ such that $
\e_k  \leq  \delta \lambda_k^2 / ( 4C + \rho N^{-1/2}\lambda_k)$, we infer that
$$
\cJ_{\cS_N}(\btheta^{k+1}) - \cJ_{\cS_N}(\btheta^k) \leq \left(C+\delta - \frac \rho 2 \right) \lambda_k^2.
$$
As a result, fixing $\rho > 2(C+\delta)$,
$$
\cJ_{\cS_N}(\btheta^{k+1}) - \cJ_{\cS_N}(\btheta^k) \leq - \tilde \kappa \lambda^2_k, \quad \tilde \kappa \coloneqq \big\vert C+\delta - \frac \rho 2 \big\vert >0 .
$$
Moreover, we can rearrange and sum over $k=0$ to $K$ the above expression to get
$$
\sum_{k=0}^K \lambda^2_k \leq \tilde \kappa^{-1} \left( \cJ_{\cS_N}(\btheta^0) - \cJ_{\cS_N}(\btheta^{K+1}) \right) \leq \tilde \kappa^{-1} \left( \cJ_{\cS_N}(\btheta^0) - \inf_{\btheta\in\st}\cJ_{\cS_N}(\btheta) \right).
$$
Therefore $\sum_{k=0}^K \lambda^2_k < +\infty$ which implies that $\lambda^2_k\to_{k\to \infty} 0$. Therefore, the \AMSA converges to a solution of the extended PMP.

Let us consider now the general case $0<\gamma_k\leq 1$. From the maximization step \eqref{eq:inexact-max-step},
\begin{equation}
\frac 1 N \sum_{i=1}^N
\wH(t, u_t^{\btheta^k, i, \e_k}, p_t^{\btheta^k, i, \e_k} , \theta^{k+1}_t)
\geq
\gamma_k \max_{\theta \in \Theta} \frac 1 N \sum_{i=1}^N  \wH(t, u_t^{\btheta^k, i, \e_k}, p_t^{\btheta^k, i, \e_k} , \theta),\quad \forall t \in [0,T].
\end{equation}
As a result, bound \eqref{eq:bound-Delta} receives an additional term and becomes,
\begin{align}
&-\, \frac 1 N \sum_{i=1}^N \int_0^T \Delta H^{i, \e_k}_{\btheta^{k+1}, \btheta^k}(t)\dt \\
& \leq
- \frac{\rho}{2N} \sum_{i=1}^N \int_0^T
\left(
\Vert \dot u_t^{\btheta^k, i, \e_k} - f( u_t^{\btheta^k, i, \e_k}, \theta^{k+1}_t) \Vert^2
+
\Vert \dot p_t^{\btheta^k, i, \e_k} - \nabla_u H(t, u_t^{\btheta^k, i, \e_k}, p_t^{\btheta^k, i, \e_k},  \theta^{k+1}_t) \Vert^2 
\right)\dt  \\
&+
\frac{\rho}{2N} \sum_{i=1}^N \int_0^T
\left(
\Vert \dot u_t^{\btheta^k, i, \e_k} - f( u_t^{\btheta^k, i, \e_k}, \theta^{k}_t) \Vert^2
+
\Vert \dot p_t^{\btheta^k, i, \e_k} - \nabla_u H(t, u_t^{\btheta^k, i, \e_k}, p_t^{\btheta^k, i, \e_k},  \theta^{k}_t) \Vert^2 
\right)\dt \\
&+
\frac{\gamma_k^{-1}-1}{N} \big\vert \sum_{i=1}^N
\wH(t, u_t^{\btheta^k, i, \e_k}, p_t^{\btheta^k, i, \e_k} , \theta^{k+1}_t) \big\vert.
\end{align}
It follows that bound \eqref{eq:bound-Delta-gamma-1} receives the same additional term
\begin{align}
&-\, \frac 1 N \sum_{i=1}^N \int_0^T \Delta H^{i, \e_k}_{\btheta^{k+1}, \btheta^k}(t)\dt
 \leq -\frac{\rho}{2}\lambda_k^2 +  \rho N^{-1/2} \e_k \lambda_k
+
\frac{\gamma_k^{-1}-1}{N} \big\vert\sum_{i=1}^N
\wH(t, u_t^{\btheta^k, i, \e_k}, p_t^{\btheta^k, i, \e_k} , \theta^{k+1}_t) \big\vert.
\end{align}
Therefore \eqref{eq:diff-J} becomes
\begin{align}
\cJ_{\cS_N}(\btheta^{k+1}) - \cJ_{\cS_N}(\btheta^k) 
&\leq
\left(C - \frac{\rho}{2}\right) \lambda_k^2 + \e_k (4C +  \rho N^{-1/2} \lambda_k)
+
\frac{\gamma_k^{-1}-1}{N} \big\vert \sum_{i=1}^N
\wH(t, u_t^{\btheta^k, i, \e_k}, p_t^{\btheta^k, i, \e_k} , \theta^{k+1}_t) \big\vert.
\end{align}
Choosing $\e_k$ like before and setting
$$
\gamma_k \geq
\frac{\big\vert \sum_{i=1}^N
\wH(t, u_t^{\btheta^k, i, \e_k}, p_t^{\btheta^k, i, \e_k} , \theta^{k+1}_t)/N \big\vert}{\delta \lambda_k^2 + \big\vert \sum_{i=1}^N
\wH(t, u_t^{\btheta^k, i, \e_k}, p_t^{\btheta^k, i, \e_k} , \theta^{k+1}_t)/N \big\vert}
$$
yields
$$
\cJ_{\cS_N}(\btheta^{k+1}) - \cJ_{\cS_N}(\btheta^k) \leq \left(C+ 2\delta - \frac \rho 2 \right) \lambda_k^2.
$$
We conclude along the same lines as before by fixing now $\rho > 2(C+2\delta)$, and inferring that
$$
\cJ_{\cS_N}(\btheta^{k+1}) - \cJ_{\cS_N}(\btheta^k) \leq - \kappa \lambda^2_k, \quad \kappa \coloneqq \big\vert C+2\delta - \frac \rho 2 \big\vert >0 .
$$
\end{proof}

Before going to the next section, a few remarks on Theorem \ref{th:conv-AMSA} are in order:
\begin{itemize}
\item Note that since the sequence of controls $(\btheta^k)$ built by \AMSA is a minimizing sequence, it  need not convergence in $\st$ norm, nor does it necessarily contain a convergent subsequence. To guarantee that the minimizing sequence converges in norm, it is necessary to additional convexity and/or compactness properties.
\item The convergence result is \emph{a priori} in the sense that it cannot be directly implemented for the following reasons:
\begin{itemize}
\item The bounds for $(\e_k)$ and $(\gamma_k)$ involve a constant $C$ which is difficult to estimate in practice.
\item The bounds may be suboptimal due to the construction of the proof.
\item In practice, it is difficult to guarantee that the maximization step is performed within a certain fraction $\gamma_k$ of the actual maximum.
\end{itemize}
\item The main interest of the result lies in the fact that it reveals the necessity to tighten the accuracy in the time integration schemes and the maximization algorithms across the iterations in order to approach the continuous sampled optimal control problem \eqref{eq:ocs}.
\end{itemize}

\section{Generalization error and convergence towards the fully continuous problem}
\label{sec:sampling-err}

The goal of this section is to connect the controls $(\btheta^k)$ of the \AMSA Algorithm \ref{alg:amsa} with the PMP solutions of the fully continuous problem \eqref{eq:oc}. For this, we rely on recent results from \cite{EHL2019} connecting the sampled PMP with the fully continuous PMP which we next briefly recall.

We assume in the following that $\btheta^*$ and $\bvtheta^*_{\cS_N}$ are solutions of \eqref{eq:oc} and \eqref{eq:ocs} such that the Hamiltonian step attains a maximum in the \emph{interior} of $\Theta$. Consequently, the continuous PMP solution $\btheta^*$ satisfies
$$
\F(\btheta^*)_t \coloneqq \bE_\mu \nabla_\theta H(t, u^{\btheta^*}_t, p^{\btheta^*}_t, \theta^*_t ) = 0
$$
for a.e. $t\in [0, T]$ and where $\F:\st \to L^\infty ([0,T], \bR^m)$. Similarly, an interior solution $\bvtheta^*_{\cS_N}$ of the sampled PMP is a random variable which satisfies
$$
\F_{\cS_N}(\bvtheta^*_{\cS_N})_t \coloneqq \frac 1 N \sum_{i=1}^N \nabla_\theta H(t, u^{\bvtheta^*_{\cS_N}, i}_t, p^{\bvtheta^*_{\cS_N}, i}_t, \vartheta_{N,t} ) = 0.
$$

Under these assumptions, the theorem below, proven in \cite{EHL2019}, describes the convergence of an interior solution $\bvtheta^*_{\cS_N}$ of the first order condition of the sampled PMP to an interior solution $\btheta^*$ of the continuous PMP. The additional local strong concavity assumption on the Hessian allows to guarantee that $\bvtheta^*_{\cS_N}$ is a global/local maximum of the sampled PMP. The result also guarantees convergence of loss function values.

\begin{definition}
For $\rho>0$ and $x\in \st$, define $S_\rho(x) \coloneqq \{ y \in \st \, : \, \Vert x-y \Vert \leq \rho \}$. The mapping $\F$ is said to be stable on $S_\rho(x)$ if there exists a constant $K_\rho>0$ such that for all $y, z\in S_\rho(x) $,
$$
\Vert y -z \Vert_{\st} \leq K_\rho \Vert \F(x) - \F(y) \Vert_{L^\infty ([0,T], \bR^m)}.
$$
\end{definition}

\begin{theorem}[Theorem 6, Corollary 1 and 2 from \cite{EHL2019}]
\label{th:impact-sampling}
Let  $\F$ be a mapping  which is stable on $S_\rho(\btheta^*)$ for some $\rho>0$ and $\btheta^*$ a solution of $\F=0$. Then there exists positive constants $s_0, C, K_1, K_2$ and $\rho_1<\rho$ and a random variable $\bvtheta^*_{\cS_N} \in S_{\rho_1}(\btheta^*)$ such that
\begin{equation}
\bP[ \Vert \btheta^* - \bvtheta^*_{\cS_N} \Vert \geq Cs ] \leq 4\exp\left(  - \frac{Ns^2}{K_1+K_2 s} \right),\quad s \in (0, s_0]
\end{equation}
and $\bvtheta^*_{\cS_N} \to \btheta^*$ as $N\to \infty$ in probability. Moreover, there exists constants $K_1', K_2'$ such that
\begin{equation}
\label{eq:bound-mean-field}
\bP[ \big\vert  \cJ(\btheta^*) - \cJ(\bvtheta^*_{\cS_N}) \big\vert \geq s] \leq 4\exp\left(  - \frac{Ns^2}{K_1'+K_2' s} \right),\quad s \in (0, s_0]
\end{equation}
In addition, if for all $t \in [0, T]$, the Hessian $\bE_{\mu}\nabla^2_{\theta, \theta} H(u^{\btheta^*}_t, p^{\btheta^*}_t, \theta^*_t ) + \lambda_0 I$ is negative definite, then $\bvtheta^*_{\cS_N}$ is also a strict local maximum of the sampled Hamiltonian $\vartheta \mapsto \frac 1 N \sum_{i=1}^N H(u_t^{\bvtheta^*_{\cS_N}, i}, p_t^{\bvtheta^*_{\cS_N}, i}, \vartheta )$. In particular, if the sampled Hamiltonian has a unique maximizer, then $\bvtheta^*_{\cS_N}$ is a solution of the sampled PMP \eqref{eq:max-PMP-classic} with the same high probability.
\end{theorem}

\begin{corollary}
\label{cor:final-conv-result}
If the sampled Hamiltonian has a unique maximizer $\bvtheta^*_{\cS_N}$, then there exists constants $\tilde K_1$ and $\tilde K_2$ such that for any $\e>0$, there exists $k(\e) \in \bN$ such that for $k\geq k(\e)$, the control $\btheta^k$ at iteration $k$ of \AMSA satisfies
$$
\bP[ | \cJ_{\cS_N}(\btheta^k) - \cJ(\btheta^*)|\geq 3 \e ] \leq 4\exp\left(  - \frac{N\e^2}{\tilde K_1+ \tilde K_2 \e} \right)
$$
\end{corollary}

\begin{proof}
We use that
\begin{align}
\bP[ \vert \cJ_{\cS_N}(\btheta^k) - \cJ(\btheta^*) \vert \geq 3\e ]
&\leq 
\bP[ \vert \cJ_{\cS_N}(\btheta^k) - \cJ_{\cS_N}(\bvtheta^*_{\cS_N}) \vert \geq \e]
+ \bP[ \vert \cJ_{\cS_N}(\bvtheta^*_{\cS_N}) - \cJ(\bvtheta^*_{\cS_N}) \vert \geq \e] \\
&+ \bP[\vert \cJ(\bvtheta^*_{\cS_N}) - \cJ(\btheta^*) \vert \geq \e],
\end{align}
and then bound each term as explained next. Let $\e>0$ be fixed. With high probability, $\bvtheta^*_{\cS_N}$ is a solution to the sampled PMP \eqref{eq:ocs}, that is $\cJ_{\cS_N}(\bvtheta^*_{\cS_N})=J^*_{\cS_N}$. By Theorem \ref{th:conv-AMSA}, there exists $k(\e) \in \bN$ such that for $k\geq k(\e)$, the control $\btheta^k$ at iteration $k$ of \AMSA satisfies $\vert \cJ_{\cS_N}(\btheta^k) - \cJ_{\cS_N}(\bvtheta^*_{\cS_N}) \vert \leq \e$, which bounds the first term. For the second term, by the infinite dimensional Hoeffding's inequality (see \cite[Corollary 2]{PS1986}), $ \vert \cJ_{\cS_N}(\bvtheta^*_{\cS_N}) - \cJ(\bvtheta^*_{\cS_N}) \vert \leq \e$ with probability $1-\exp(-N\e^2 / (\tilde K_1' +\tilde K_2' \e))$. The third term is also bounded by $\e$ with probability $4\exp\left(  - \frac{N\e^2}{K_1'+K_2' \e} \right)$ thanks to \eqref{eq:bound-mean-field}.
\end{proof}

\section{Guidelines for numerical implementation}
\label{sec:impl}
Algorithm \ref{alg:amsa} should in principle be implemented with tolerances $\e_k$ and $\gamma_k$ given in Theorem \ref{sec:conv-AMSA} to ensure convergence. These tolerances are in practice hard to guarantee, and the convergence analysis that leads to these quantities may be suboptimal. A trade-off between the theory and numerical implementation consists simply in replacing the forward and backward propagations at prescribed accuracies by propagations where the time steps are progressively refined across the iterations. It is therefore necessary to prescribe a refinement strategy. For the Hamiltonian maximization step (line 13 of Algorithm \ref{alg:amsa}), the best that one can do is to use an efficient optimizer of nonconvex problems.

In addition to these considerations, note that Algorithm \ref{alg:amsa} can be parallelized at several parts: for each sample $i$, the forward and backward propagations can be performed in parallel (see the for loop of line 8). Each propagation can additionally be parallelized by means of parallel in time algorithms as in \cite{GRSCG2020} (see \cite{LMT2001, MM2020} for some selected references on the topic). In addition, the Hamiltonian maximization step can also be parallelized (see \texttt{for} loop of line 12 of Algorithm \ref{alg:amsa}).
For our concern, we have chosen to parallelize the Maximization step, since it is way more expensive in terms of computational time than the propagation step.

\section{Numerical experiments}
\label{sec:numerics}
In this section, we present some numerical experiments, aiming primarily at illustrating the behavior of the \AMSA algorithm. We give special focus on studying the performance in the learning phase in terms of the value of the loss function.
We also discuss computing times and examine expressivity in terms of generalization errors.
Note that the final quality of approximation depends on the number of samples, on the network depth, and on the optimizers and their starting guesses.
We proceed by increasing levels of difficulty in the numerical examples in order to disentangle the effect of each of the different aspects (note however that their effects are usually mixed in non-trivial applications).
The code to reproduce the results can be found at 
\begin{center}
\url{https://github.com/jaghili/amsa}.
\end{center}
The implementation has been done with Python 3 and NumPy. Particular attention has been paid to vectorize most of the operations in order to delegate as much as possible the looping to internal, highly optimized C and Fortran functions. We have observed that the time to perform the ODE propagations is negligible with respect to the Hamiltonian maximization (line 12 of Algorithm \ref{alg:amsa}). We have thus parallelized this step with MPI.

%
% Figure template functions
%
\newcommand{\boxplot}[3]{%1=info, %2=test-case, %3=strategy
  \begin{subfigure}[b]{0.31\textwidth}
    \centering
    \includegraphics[width=\textwidth]{data/min_k_#1_#2_N_#3.png}
    \caption{$\min_{1\leq k \leq k_{\max}} \cJ^{\text{(run=r)}}_{\cS_N}(\btheta^k)$.}
    \label{fig:min_k_#1_#2_N_#3}
  \end{subfigure}
}

\newcommand{\boxplottest}[3]{%1=info, %2=test-case, %3=strategy
  \begin{subfigure}[b]{0.31\textwidth}
    \centering
    \includegraphics[width=\textwidth]{data/min_k_#1_#2_N_#3.png}
    \caption{$\min_{1\leq k \leq k_{\max}} \cJ^{\text{(run=r)}}_{\widetilde\cS_N}(\btheta^k)$.}
    \label{fig:min_k_#1_#2_N_#3}
  \end{subfigure}
}

\newcommand{\lossplot}[3]{%1=info, %2=test-case, %3=strategy
  \begin{subfigure}[b]{0.31\textwidth}
    \centering
    \includegraphics[width=\textwidth]{data/#1_#2_#3.png}
    \caption{#3}
    \label{fig:#1_#2_#3}
  \end{subfigure}
}

\newcommand{\legendwidth}{250pt}

%
% End Figure template functions
%

\subsection{Approximation of the sine function}
In this example, inspired from \cite[Section 6]{LCTE2018}, the task is to approximate the graph of the sine function,
$$
y: X=[-\pi,\pi] \to Y = [-1,1], \qquad
x \mapsto y(x) = \sin (x).
$$
Note that here the dimensions $n$ and $k$ of the domain and label sets are $n=k=1$. The continuous learning problem (see \eqref{eq:oc}) is to minimize the $L^2(X)$ approximation error over the controls $\btheta\in\st$, namely
$$
\inf_{\btheta\in\st} \frac 1 2 \int_X | y(x) - g(u_T^{\btheta, x}) |^2 \,\dx.
$$
For every point $x\in X$, $u_T^{\btheta, x}$ is the final time state of the ODE
\begin{equation}
\begin{cases}
\dot \u^{\btheta,x}_t &= \tanh(A_t \cdot u_t^{\btheta,x} + b_t), \quad \forall t\in (0,T] \\
 u_0^{\btheta, x} &= x (1,\dots, 1)^T \in \bR^d.
\end{cases}
\end{equation}
Here, $\theta_t = (A_t, b_t)\in \bR^{d\times d}\times \bR^d \sim \bR^{d^2+d}$ are the parameters to optimize. We constraint them to lie in $[\alpha_{\min}, \alpha_{\max}]=[-1, 1]$ so $\Theta=[-1, 1]^{d^2+d}$. The function $g$ is defined as
$$
g: \bR^d \to \bR,\quad z = (z_1,\dots, z_d) \mapsto g(z) \coloneqq \frac 1 d \sum_{i=1}^d z_i,
$$
and note that we have chosen $\mu= \dx$ as the Lebesgue measure, $f =\tanh$ as the activation function, $\Phi(u_T^{\btheta, x}, y) = \tfrac{1}{2}| y(x) - g(u_T^{\btheta, x}) |^2 $ and we do not have any regularisation ($R=0$).

In the sampled version which we consider in our experiments, we work with a set $\cS_N = \{(x_i, y_i)\}_{i=1}^N$ of pairs generated with equidistant data points
$$
x_i = -\pi + 2\pi  \frac{i-1}{N-1} 
, \quad y_i = \sin (x_i),\quad i\in \{1,\dots, N\},
$$
and perform empirical risk minimization taking the uniform distribution $\mu_N = \frac 1 N \sum_{i=1}^N \delta_{(x_i, y_i)} $ (see \eqref{eq:ocs}). We solve this problem with the above described \AMSA algorithm where the final time is set to $T=5$, the penalization parameter is set to $\rho=5$, and the number of neurons per layer to $d=3$.
We use a basic explicit Euler scheme as a time integrator, which induces a ResNet architecture.
We progressively refine the time discretization from $L=3$ to $L=32$ time steps (in other words, we increase the network depth from 3 to 32 layers) using three different strategies:
\begin{itemize}
\item[(A1)] \emph{Abrupt refinement}: We refine from shallow ($L=3$) to deep ($L=32$) at iteration $k=250$.
\item[(A2)] \emph{Fast refinement}: We add $10$ layers every $50$ iterations.
\item[(A3)] \emph{Slow refinement}: We add $10$ layers every $100$ iterations.
\end{itemize}
Our goal is to compare the behavior of these refinement strategies with the non-adaptive training of a shallow and a deep network having $L=3$ and $L=32$ layers respectively.

At each iteration $k\geq0$ of \AMSA, the maximization step is performed with NumPy's L-BFGS-B optimizer, which is an iterative algorithm suited for nonconvex optimization problems.
The optimizer is initialized by picking the best candidate among a list $\Theta^{0}$ of vectors, including the last best control $\btheta^{k}_{\text{best}}=\argmin_{0\le i < k} \cJ_{\cS_N}(\btheta^i)$ among all previous iterations before $k$ and random perturbations of it, precisely
$$
\Theta^{0} \coloneqq \bigcup_{q=0}^5 \bigcup_{i=1}^{25} \left\{ \btheta^k_{\text{best}} + 10^{-2q}\mathtt{rand}(\alpha_{\min},\alpha_{\max}), 10^{-2q}\mathtt{rand}(\alpha_{\textrm{min}}, \alpha_{\textrm{max}}) \right\}
$$
where $\mathtt{rand}(\alpha_{\min},\alpha_{\max})$ is a uniform real distribution function in the interval $[\alpha_{\min}, \alpha_{\max}]$.
As a result, the final output of the algorithm is random due to the choice of the initial guess.
For this reason, we make $\bar r = 20$ runs of \AMSA, each one consisting in  $k_{\max}=800$ iterations.

%For the approximations of out-of-training data, we choose the best available control $\hat \btheta$, which is the one that reaches the minimal value over all iterations and over all $\bar r$ repetitions. It satisfies
%$$
%\cJ_N(\hat \btheta) \in 
%\min_{1\leq r \leq \bar r} \min_{1\leq k \leq k_{\max}} \cJ^{\text{(run=r)}}_{N}(\btheta^k).
%$$

In Figure \ref{fig:loss_Sine}, we first fix $N=20$ samples and we plot the convergence history of the train loss $\cJ^{\text{(run=r)}}_{\cS_N}(\btheta^k)$ (see red colors). We also plot the generalization error via the \textit{test loss} across the iterations $k$ (see blue colors). The test loss is the quantity $\cJ_{\widetilde\cS_N}(\btheta^k)$, defined in \eqref{eq:ocs}, where we use a test set $\widetilde\cS_N$ different from the training set $\cS_N$. $\widetilde\cS_N$ is taken uniformly random at each iteration $1\le k \le \kmax$.
Since the history depends on each run $r$, the figure shows some statistics related to the repetitions. The continuous red curve shows the average value $\cJ^{(\text{av})}_{\cS_N, k} \coloneqq  (1/\bar r)\sum_{r=1}^{\bar r} \cJ^{\text{(run=r)}}_{\cS_N}(\btheta^k)$ over the runs and the red diffuse color shows its distribution around the average. The same presentation is given in blue for the generalization errors. We can first observe that the convergence history is, in average, relatively similar for all architectures. The average value of the loss function does not go below $10^{-2}$ for the shallow network (figure \ref{fig:loss_Sine_Shallow}), and the value reached by the deep network is only slightly better (figure \ref{fig:loss_Sine_Deep}): at the end of the iterations, the cost function is in average below $7.10^{-2}$. Similar values are reached by any of our three adaptive strategies (see figures \ref{fig:loss_Sine_A1}, \ref{fig:loss_Sine_A2} and \ref{fig:loss_Sine_A3}). As a consequence, if we take this average value as the performance indicator, we may conclude that the approximation power of the deep network is only marginally better than the shallow one. We may also conclude that the adaptive training strategy reduces to some extend the computational time for training (as we illustrate further on in section \ref{sec:runtimes}) but it does not bring extra approximation power. However, note that this indicator is not entirely appropriate because it does not correspond to any realized convergence history and, more importantly, because it does not inform about the best performance that we have at hand. It is crucial to remark that the variance around the average convergence value is significantly larger in the adaptive strategies compared to the non adaptive ones. This is particularly true for (A2) and (A3). Figure \ref{fig:min_k_trainloss_Sine_N_20} shows that the minimal value $\min_{1\leq k \leq k_{\max}} \cJ^{\text{(run=r)}}_{\cS_N}(\btheta^k)$ reached during the $k_{\max}=800$ iterations of each run is, in average, one order of magnitude lower than the one of the non-adaptive coarse and the non-adaptive deep network. This result illustrates that the adaptive strategy significantly contributes to reach better minima, thus allowing to benefit better from the higher approximation power of the deeper neural network.

Figure \ref{fig:pred_Sine} shows the reconstruction of the sine function with the controls performing best in each training run, namely, realizing $\min_{1\leq k \leq k_{\max}} \cJ^{\text{(run=r)}}_{\cS_N}(\btheta^k)$ for every $r$. We see that the best over all realizations yields a very satisfactory approximation. Statistics on the generalization errors of these best controls are given in Figure \ref{fig:min_k_testloss_Sine_N_20}

\begin{figure}
  \centering
  \includegraphics[width=\legendwidth]{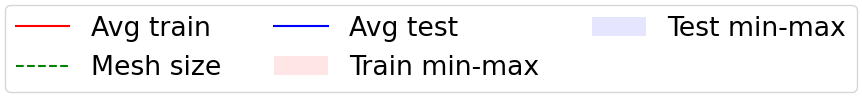}
  \\  
\lossplot{loss}{Sine}{Shallow}
\lossplot{loss}{Sine}{Deep}
\boxplot{trainloss}{Sine}{20}
\\
\lossplot{loss}{Sine}{A1}
\lossplot{loss}{Sine}{A2}
\lossplot{loss}{Sine}{A3}
\caption{Sine function. Training loss $\cJ^{\text{(run=r)}}_{\cS_N}(\btheta_k)$ (red) and test loss $\cJ^{\text{(run=r)}}_{\widetilde\cS_N}(\btheta_k)$ (blue). Here $N=20$ and $\bar{r}=20$ runs.}
\label{fig:loss_Sine}
\end{figure}

\begin{figure}
  \centering
      \includegraphics[width=\legendwidth]{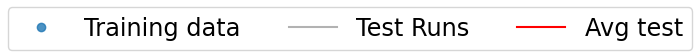}
    \\  
\lossplot{pred}{Sine}{Shallow}
\lossplot{pred}{Sine}{Deep}
\boxplottest{testloss}{Sine}{20}
\\
\lossplot{pred}{Sine}{A1}
\lossplot{pred}{Sine}{A2}
\lossplot{pred}{Sine}{A3}
\caption{Sine function: Predictions of $\bar r=20$ runs and average prediction.}
\label{fig:pred_Sine}
\end{figure}

We can next examine in Figure \ref{fig:sinus} the impact of the number $N$ of data samples in the generalization errors of the best controls. As one can expect, the generalization errors decreases when the number $N$ of samples increases. For $N\geq 20$, the errors with the shallow network stagnate before reaching $10^{-3}$ in average regardless of the number of samples. This illustrates the limitations in the approximation power of the shallow network in the current example. The non-adaptive deep network ($L=32$) presents a behavior which is only better for a reduced number of samples $N\leq 8$ but overall its performance is relatively similar to the coarse neural network. This comes from the difficulty in finding good quality optimizers in networks involving many coefficients. If we now adaptively train the deep network with strategy (A2), the generalization errors are in general lower than the ones reached by the non-adaptive strategies, especially when $N\geq 15$. The result illustrates that the adaptive strategy gives in average a better performance in terms of approximation. This is obtained at a reduced computing time as we show further on.

\newcommand{\figureheight}{120pt}
\begin{figure}
  \centering
 \begin{subfigure}[b]{0.31\textwidth}
   \includegraphics[width=\textwidth]{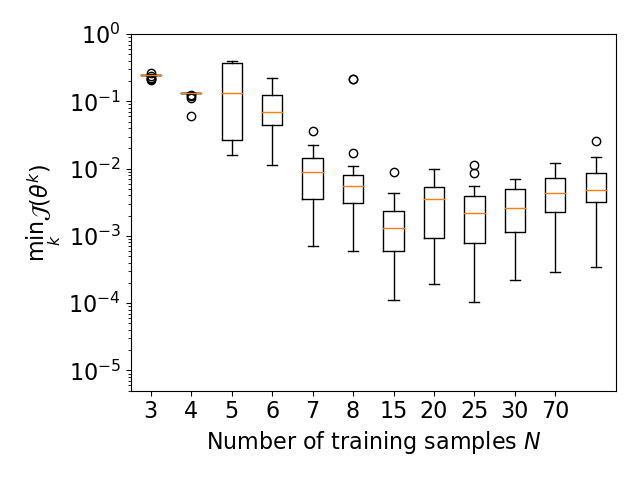}
   \caption{Shallow: $3$ layers}
   \label{fig:sinus.coarse}
 \end{subfigure}
 \begin{subfigure}[b]{0.31\textwidth}
   \includegraphics[width=\textwidth]{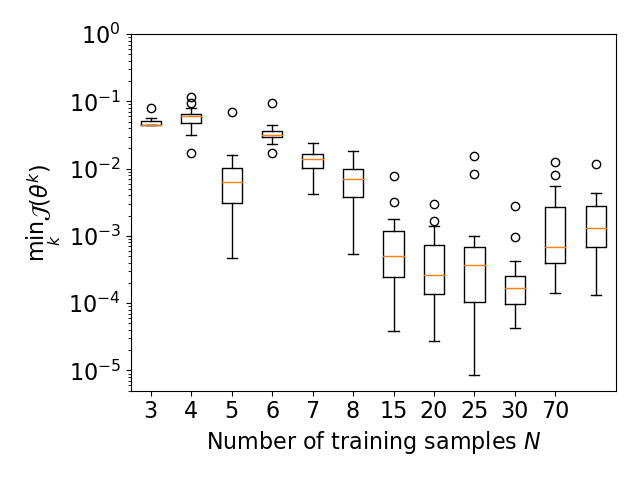}
   \caption{A2: $3\to 32$ layers.}
   \label{fig:sinus.adapt}
 \end{subfigure}
 \begin{subfigure}[b]{0.31\textwidth}
   \includegraphics[width=\textwidth]{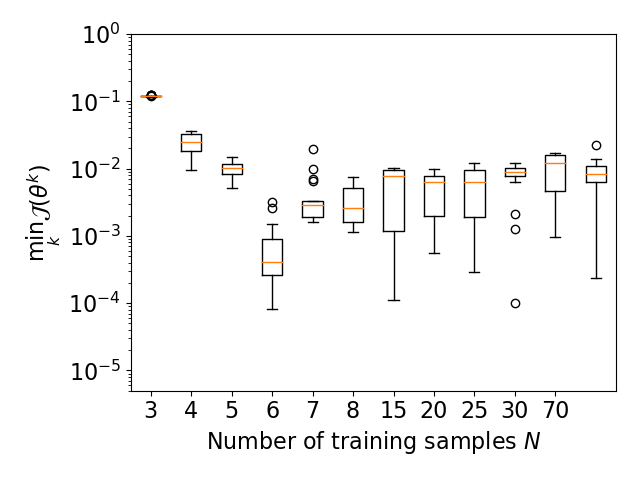}
   \caption{Deep: $32$ layers}
   \label{fig:sinus.thin}
 \end{subfigure}
  \caption{Sine case -- Generalization errors vs $N$ for coarse, adaptive (A2) and thin networks}
  \label{fig:sinus}
\end{figure}

\subsection{Noisy step function}
We next examine the robustness of the method against noise.We consider the graph of the step function,
$$
y: X=[-1,1] \to Y = [-0.5,0.5], \qquad x \mapsto y(x) = 
\begin{cases}
0.5 &\quad \text{if } x\leq 0 \\
-0.5 &\quad \text{if } x> 0 \\
\end{cases}
$$
and we seek to approximate it when the given data are noisy. For each sample $i$, we get the pair $(x_i, y_i)$ with
$$
y_i = y(x_i) + \varepsilon_i,\quad i=1,\dots,N,
$$
and the $\varepsilon_i$ follow a uniform distribution in $[-0.2, 0.2]$.

We consider the same setting as before ($T=5,\,\rho=5$, $d=3$) and fix $N=800$. Figure \ref{fig:loss_Step} shows the convergence history of the training with the shallow, deep and adaptively deep neural networks. We consider the same adaptive strategies as before. All approaches reach more or less the same minimal value in the training (see Figure \ref{fig:min_k_trainloss_Step_N_20}. However, we observe that (A2) and (A3) give better generalization errors than  the non-adaptive deep architecture (see Figure \ref{fig:min_k_testloss_Step_N_20}). Note however that this superiority is not very large and all methods produce excellent reconstruction results as Figure \ref{fig:pred_Step} illustrates. We think that this is due to the fact that in this case all methods find configurations that are very close the global optimum, which is here given by the approximation error of the mapping $x\mapsto y(x)$.

\begin{figure}
  \centering
    \includegraphics[width=\legendwidth]{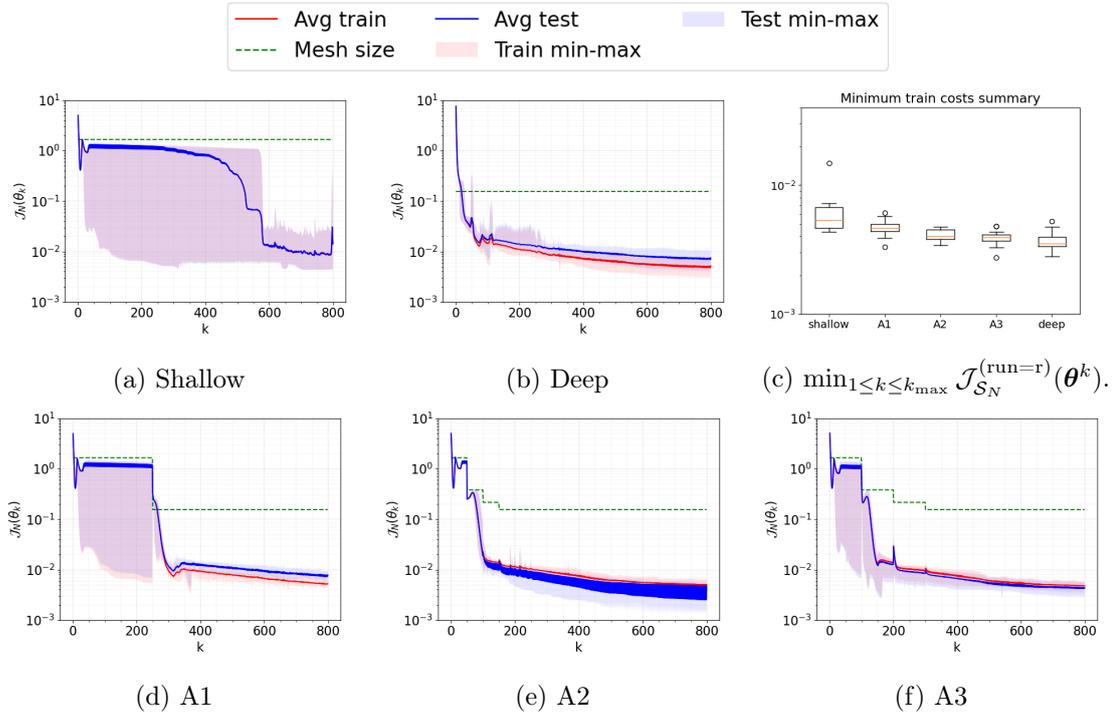}
    \\  
\lossplot{loss}{Step}{Shallow}
\lossplot{loss}{Step}{Deep}
\boxplot{trainloss}{Step}{20}
\\
\lossplot{loss}{Step}{A1}
\lossplot{loss}{Step}{A2}
\lossplot{loss}{Step}{A3}
\caption{Noisy step: Loss $\cJ^{\text{(run=r)}}_N(\btheta_k)$ for $N=800$ and $\bar{r}=20$ runs.}
\label{fig:loss_Step}
\end{figure}

\begin{figure}
  \centering
    \includegraphics[width=\legendwidth]{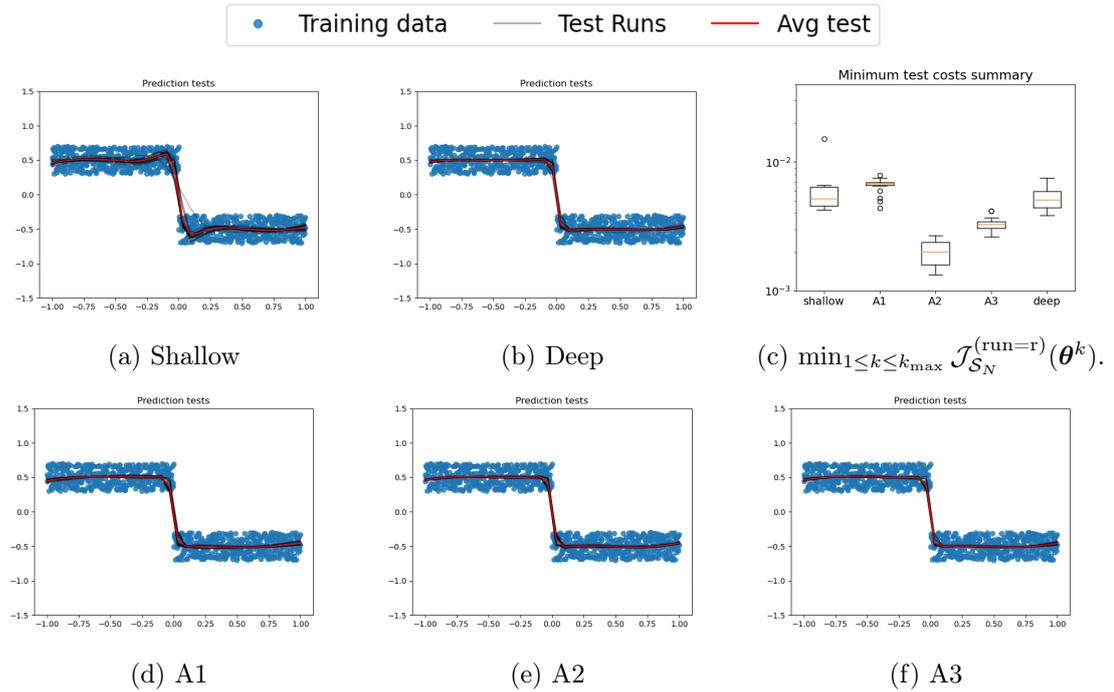}
    \\  
\lossplot{pred}{Step}{Shallow}
\lossplot{pred}{Step}{Deep}
\boxplot{testloss}{Step}{20}
\\
\lossplot{pred}{Step}{A1}
\lossplot{pred}{Step}{A2}
\lossplot{pred}{Step}{A3}
\caption{Noisy step: Predictions of some runs.}
\label{fig:pred_Step}
\end{figure}

\subsection{Classification}
\label{sec:classif}
As a last example, we consider a simple 2D classification problem where the function to approximate is
$$
y: X=[-1,1]^2 \to Y = \{0,1\}, \qquad (x_1,x_2) \mapsto y(x_1,x_2) = 
\begin{cases}
1 &\quad \text{if } x_1^2+x_2^2 \leq (0.5)^2, \\
0 &\quad \text{otherwise.}
\end{cases}
$$
The setting is the same as in the above examples but with slight changes.
Each layer has now $d=2\times 3=6$ neurons per layers, the constraint on the controls is set to $[\alpha_{\min}, \alpha_{\max}]=[-2,2]$ and  the output function $g$ is defined as
$$
g: \bR^d \to \bR,\quad z = (z_1,\dots, z_d) \mapsto g(z) \coloneqq  \mathcal{H}\left(\frac 1 d \sum_{i=1}^d z_i\right)
$$
where $\mathcal{H}$ is the usual Heaviside step function.
The figure layout of the results is the same as in the other examples.
Figures \ref{fig:loss_Classif} and \ref{fig:pred_Classif} show statistics on the loss function and examples of reconstruction.
The blue-to-red scalar field in Fig \ref{fig:pred_Step} is the output of the trained network composed with a treshold filter $\mathcal{F}:x \mapsto \mathcal{H}(x-0.5)$ from $1024$ inputs points.
The black crosses ($\times$) and white dots ($\circ$) are the $N=800$ training points used for each case, the first refers to points with $z=0$, the second to $z=1$.
Essentially, we observe that \AMSA yields the same quality of approximation as the non-adaptive deep neural network.
{
\begin{figure}[H]
  \centering
    \includegraphics[width=\legendwidth]{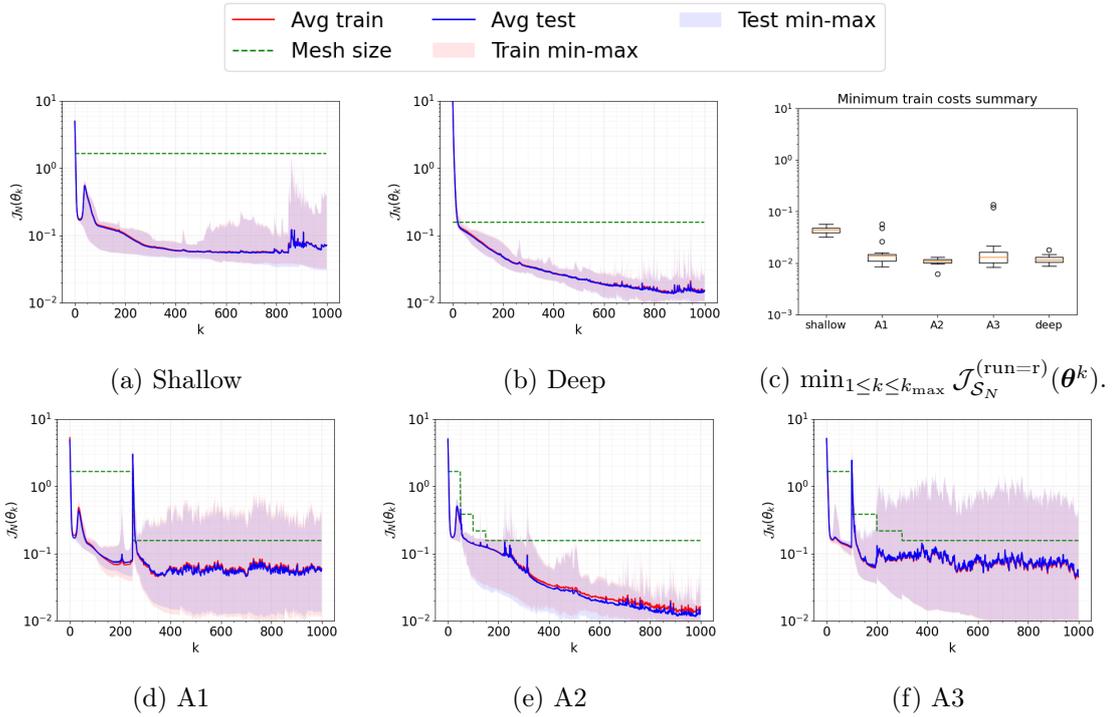}
    \\  
    \lossplot{loss}{Classif}{Shallow}
    \lossplot{loss}{Classif}{Deep}
    \boxplot{trainloss}{Classif}{800}
    \\
    \lossplot{loss}{Classif}{A1}
    \lossplot{loss}{Classif}{A2}
    \lossplot{loss}{Classif}{A3}
\caption{Classification: Loss $\cJ^{\text{(run=r)}}_N(\btheta_k)$ for $N=800$ and $20$ runs.}
\label{fig:loss_Classif}
\end{figure}

\begin{figure}
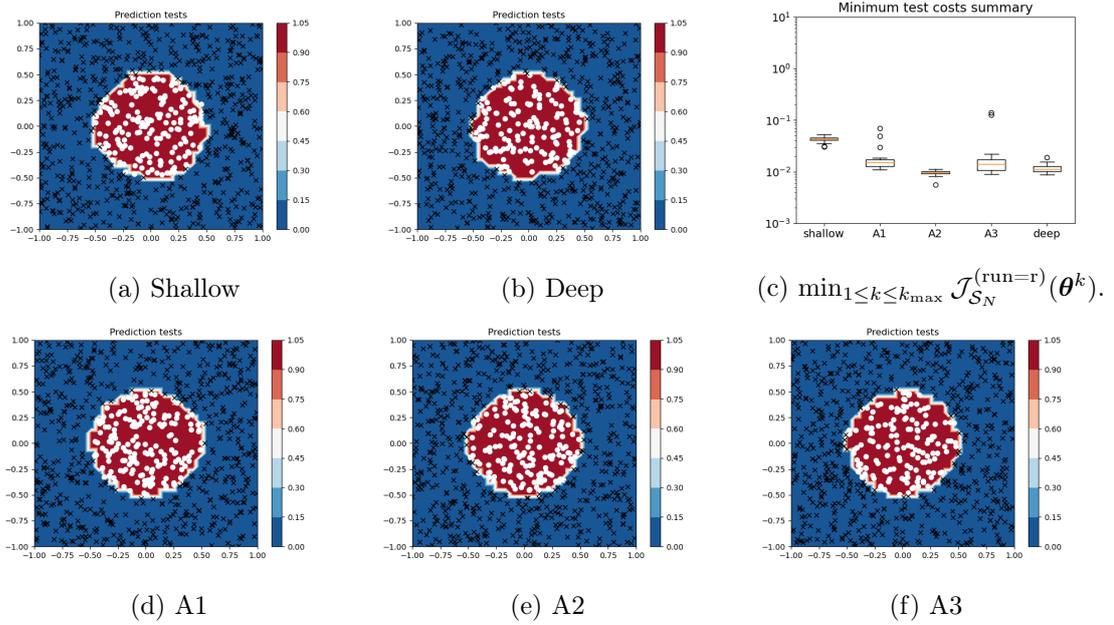

  \centering
\lossplot{pred}{Classif}{Shallow}
\lossplot{pred}{Classif}{Deep}
\boxplot{testloss}{Classif}{800}
\\
\lossplot{pred}{Classif}{A1}
\lossplot{pred}{Classif}{A2}
\lossplot{pred}{Classif}{A3}
\caption{Classification: Predictions of $\bar r=20$ runs and average prediction. Data labels: $\boxed{\times: 0}$, $\boxed{\circ: 1}$.}
\label{fig:pred_Classif}
\end{figure}
}

\subsection{Computational times}
\label{sec:runtimes}
We finish the section on numerical tests by examining the computational time in the learning approach. We present runtimes only for the noisy step example for the sake of brevity (similar results are obtained for the other tests). We consider three different criteria:
\begin{itemize}
\item \textbf{Complexity index (see Figure \ref{fig:runtime_Step_Sequential}):} we estimate the number of operations by measuring the runtime of a sequential run.
\item \textbf{Parallel runtime index (see Figure \ref{fig:runtime_Step_Parallel}):} we measure the runtime of running in parallel the Hamiltonian maximization step of the learning algorithm. We place ourselves in a scenario with no constraints in the computing ressources and allocate one processor per layer. Note that the number of layers is increased in the adaptive approach so we use more and more ressources as the algorithm makes refinements.
\item \textbf{Energy consumption index (see Figure \ref{fig:runtime_Step_Sequential}):} we estimate the total energy consumption of a parallel run as follows. For each iteration of the algorithm, we add the corresponding computing times of all the processors involved in that iteration. This gives an estimate in \texttt{cpu.seconds} of the energy consumption for each iteration. By summing over all iterations, we obtain the final estimate. This quantity is in fact equal to the cumulative sequential runtime which is given on Figure \ref{fig:runtime_Step_Sequential}.
\end{itemize}
Our runtimes have been measured on a cluster and we have observed that its occupation greatly impacts on the runtime results. To illustrate this issue, we show runtimes for the training of the shallow neural network in an empty node and a busy node, i.e. a node where others jobs are running in parallel. We have run the rest of the examples on empty nodes. % but, as we discuss next, some unexpected behaviors are still observed.

If we first consider Figure \ref{fig:runtime_Step_Sequential}, we see that the adaptive training strategy performs better than the non-adaptive one regarding the sequential runtimes and the energy consumption. This is comes as no surprise since the adaptive strategy performs less operations at the beginning of the iterations since they involve less layers. When computations are run in parallel, Figure \ref{fig:runtime_Step_Parallel}) shows that the run times are all very similar for all strategies. This is due to the fact that the tasks are well-balanced between processors since each of them does computations for one layer. We also observe that runtimes are greatly affected by the occupation of the cluster.

\begin{figure}
\centering
\begin{subfigure}[b]{0.45\textwidth}
    \centering
    \includegraphics[width=\textwidth]{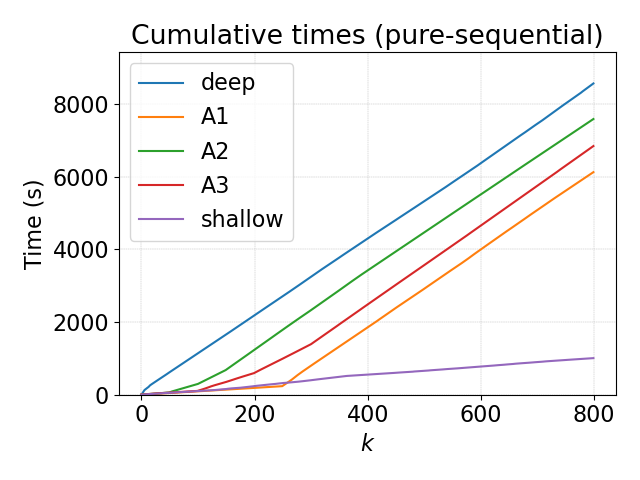}
    \caption{Complexity and energy consumption.}
    \label{fig:runtime_Step_Sequential}
\end{subfigure}
\begin{subfigure}[b]{0.45\textwidth}
    \centering
    \includegraphics[width=\textwidth]{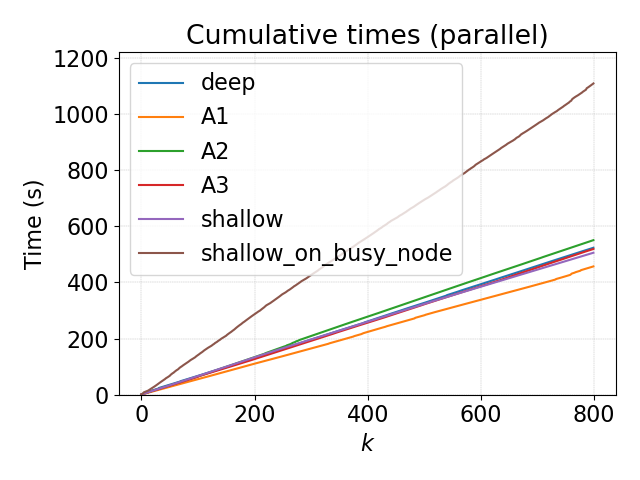}
    \caption{Parallel runtime.}
    \label{fig:runtime_Step_Parallel}
\end{subfigure}
\caption{Computational times for the noisy step example.}
    \label{fig:runtime_Step}
\end{figure}

\section{Conclusion}
\label{sec:conclusion}
We have proven that the \AMSA algorithm converges to an underlying continuous learning problem in the limit of the time discretization and of the the number of samples. The convergence analysis requires that the time propagations and Hamiltonian maximizations are performed at each step with increasingly tight accuracy tolerances. Since a sharp estimation of these tolerances is difficult to obtain, we have implemented \AMSA with certain refinement strategies in order to examine its potential. The numerical experiments reveal that the adaptive strategy helps to benefit in practice from the higher approximation properties of deep networks by mitigating over-parametrization issues: our results show that adaptivity increases the chances to find better quality minimizers compared to the non-adaptive training of deep neural networks. In addition, it appears that the adaptive strategy is clearly more performant in terms of complexity and energy consumption compared to the non-adaptive training of a deep neural network.

\section*{Acknowledgments and Disclosure of Funding}
This research has been funded by the Emergences Project of the Paris city council called ``Models and Measures''.

% Bibliography
%\bibliographystyle{bibstyle}
\bibliography{literature}

\begin{thebibliography}{44}
\providecommand{\natexlab}[1]{#1}
\providecommand{\url}[1]{\texttt{#1}}
\expandafter\ifx\csname urlstyle\endcsname\relax
  \providecommand{\doi}[1]{doi: #1}\else
  \providecommand{\doi}{doi: \begingroup \urlstyle{rm}\Url}\fi

\bibitem[Akrivis et~al.(2009)Akrivis, Makridakis, , and Nochetto]{AMN2009}
G.~Akrivis, C.~Makridakis, , and R.~Nochetto.
\newblock Optimal order a posteriori error estimates for a class of
  runge--kutta and galerkin methods.
\newblock \emph{Numerische Mathematik}, 114\penalty0 (1):\penalty0 133, Aug
  2009.
\newblock ISSN 0945-3245.
\newblock \doi{10.1007/s00211-009-0254-2}.
\newblock URL \url{https://doi.org/10.1007/s00211-009-0254-2}.

\bibitem[Athans and Falb(2013)]{AF2013}
M.~Athans and P.~L. Falb.
\newblock \emph{Optimal control: an introduction to the theory and its
  applications}.
\newblock Courier Corporation, 2013.

\bibitem[Benning et~al.(2019)Benning, Celledoni, Ehrhardt, Owren, and
  Schönlieb]{BCEOS2019}
M.~Benning, E.~Celledoni, M.~J. Ehrhardt, B.~Owren, and C.~Schönlieb.
\newblock Deep learning as optimal control problems: Models and numerical
  methods.
\newblock \emph{Journal of Computational Dynamics}, 6\penalty0 (2):\penalty0
  171--198, 2019.
\newblock ISSN 2158-2491.
\newblock \doi{10.3934/jcd.2019009}.
\newblock URL
  \url{http://aimsciences.org//article/id/37d95dff-e7a8-4aee-9ee6-cf35d0deb544}.

\bibitem[Bertsekas(1995)]{Bertsekas1995}
D.~P. Bertsekas.
\newblock \emph{Dynamic programming and optimal control}, volume~1.
\newblock Athena scientific Belmont, MA, 1995.

\bibitem[Boltyanskii et~al.(1960)Boltyanskii, Gamkrelidze, and
  Pontryagin]{BGP1960}
V.~G. Boltyanskii, R.~V. Gamkrelidze, and L.~S. Pontryagin.
\newblock The theory of optimal processes. i. the maximum principle.
\newblock Technical report, TRW SPACE TECHNOLOGY LABS LOS ANGELES CALIF, 1960.

\bibitem[Bottou(2010)]{Bottou2010}
L.~Bottou.
\newblock Large-scale machine learning with stochastic gradient descent.
\newblock In \emph{Proceedings of COMPSTAT'2010}, pages 177--186. Springer,
  2010.

\bibitem[Bressan and Piccoli(2007)]{BP2007}
A.~Bressan and B.~Piccoli.
\newblock \emph{Introduction to the mathematical theory of control}, volume~1.
\newblock American institute of mathematical sciences Springfield, 2007.

\bibitem[Chen et~al.(2015)Chen, Goodfellow, and Shlens]{CGS2016}
T.~Chen, I.~Goodfellow, and J.~Shlens.
\newblock Net2net: Accelerating learning via knowledge transfer.
\newblock \emph{arXiv preprint arXiv:1511.05641}, 2015.

\bibitem[Chernousko and Lyubushin(1982)]{CL1982}
F.~L. Chernousko and A.~A. Lyubushin.
\newblock Method of successive approximations for solution of optimal control
  problems.
\newblock \emph{Optimal Control Applications and Methods}, 3\penalty0
  (2):\penalty0 101--114, 1982.

\bibitem[Clarke(2005)]{Clarke2005}
F.~Clarke.
\newblock The maximum principle in optimal control, then and now.
\newblock \emph{Control and Cybernetics}, 34\penalty0 (3):\penalty0 709, 2005.

\bibitem[Cortes et~al.(2017)Cortes, Gonzalvo, Kuznetsov, Mohri, and
  Yang]{CGKMY2017}
C.~Cortes, X.~Gonzalvo, V.~Kuznetsov, M.~Mohri, and S.~Yang.
\newblock Adanet: Adaptive structural learning of artificial neural networks.
\newblock In \emph{Proceedings of the 34th International Conference on Machine
  Learning-Volume 70}, pages 874--883. JMLR. org, 2017.

\bibitem[Daubechies et~al.(2019)Daubechies, DeVore, Foucart, Hanin, and
  Petrova]{DDFHP2019}
I.~Daubechies, R.~DeVore, S.~Foucart, B.~Hanin, and G.~Petrova.
\newblock Nonlinear approximation and (deep) relu networks.
\newblock \emph{arXiv preprint arXiv:1905.02199}, 2019.

\bibitem[Duchi et~al.(2011)Duchi, Hazan, and Singer]{DHS2011}
J.~Duchi, E.~Hazan, and Y.~Singer.
\newblock Adaptive subgradient methods for online learning and stochastic
  optimization.
\newblock \emph{Journal of machine learning research}, 12\penalty0
  (Jul):\penalty0 2121--2159, 2011.

\bibitem[E et~al.(2019)E, Han, and Li]{EHL2019}
W.~E, J.~Han, and Q.~Li.
\newblock A mean-field optimal control formulation of deep learning.
\newblock \emph{Research in the Mathematical Sciences}, 6\penalty0
  (1):\penalty0 10, 2019.

\bibitem[Elsken et~al.(2018)Elsken, Metzen, and Hutter]{EMH2018}
T.~Elsken, J.~H. Metzen, and F.~Hutter.
\newblock Neural architecture search: A survey.
\newblock \emph{arXiv preprint arXiv:1808.05377}, 2018.

\bibitem[Gomez et~al.(2017)Gomez, Ren, Urtasun, and Grosse]{GRUG2017}
A.~Gomez, M.~Ren, R.~Urtasun, and R.~Grosse.
\newblock The reversible residual network: Backpropagation without storing
  activations.
\newblock In \emph{Advances in neural information processing systems}, pages
  2214--2224, 2017.

\bibitem[Grohs et~al.(2019)Grohs, Perekrestenko, Elbr{\"a}chter, and
  B{\"o}lcskei]{GPEB2019}
P.~Grohs, D.~Perekrestenko, D.~Elbr{\"a}chter, and H.~B{\"o}lcskei.
\newblock Deep neural network approximation theory.
\newblock \emph{arXiv preprint arXiv:1901.02220}, 2019.

\bibitem[G{\"u}hring et~al.(2019)G{\"u}hring, Kutyniok, and Petersen]{GKP2019}
I.~G{\"u}hring, G.~Kutyniok, and P.~Petersen.
\newblock Error bounds for approximations with deep {ReLU} neural networks in
  $w^{s, p}$ norms.
\newblock \emph{arXiv preprint arXiv:1902.07896}, 2019.

\bibitem[Günther et~al.(2020)Günther, Ruthotto, Schroder, Cyr, and
  Gauger]{GRSCG2020}
S.~Günther, L.~Ruthotto, J.~B. Schroder, E.~C. Cyr, and N.~R. Gauger.
\newblock Layer-parallel training of deep residual neural networks.
\newblock \emph{SIAM Journal on Mathematics of Data Science}, 2\penalty0
  (1):\penalty0 1--23, 2020.

\bibitem[Haber and Ruthotto(2018)]{HR2018}
E.~Haber and L.~Ruthotto.
\newblock Stable architectures for deep neural networks.
\newblock \emph{Inverse Problems}, 34\penalty0 (1):\penalty0 014004, 2018.

\bibitem[He et~al.(2016)He, Zhang, Ren, and Sun]{HZRS2016}
K.~He, X.~Zhang, S.~Ren, and J.~Sun.
\newblock Deep residual learning for image recognition.
\newblock In \emph{Proceedings of the IEEE conference on computer vision and
  pattern recognition}, pages 770--778, 2016.

\bibitem[Hebb(1949)]{Hebb1949}
D.~Hebb.
\newblock \emph{The organization of behavior: a neuropsychological theory}.
\newblock Wiley, 1949.
\newblock ISBN 0-8058-4300-0.

\bibitem[Krizhevsky et~al.(2012)Krizhevsky, Sutskever, and Hinton]{KSH2012}
A.~Krizhevsky, I.~Sutskever, and G.~Hinton.
\newblock Imagenet classification with deep convolutional neural networks.
\newblock In \emph{Advances in neural information processing systems}, pages
  1097--1105, 2012.

\bibitem[Larsson et~al.(2016)Larsson, Maire, and Shakhnarovich]{LMS2016}
G.~Larsson, M.~Maire, and G.~Shakhnarovich.
\newblock Fractalnet: Ultra-deep neural networks without residuals.
\newblock \emph{arXiv preprint arXiv:1605.07648}, 2016.

\bibitem[LeCun et~al.(1988)LeCun, Touresky, Hinton, and Sejnowski]{Lecun1988}
Y.~LeCun, D.~Touresky, G.~Hinton, and T.~Sejnowski.
\newblock A theoretical framework for back-propagation.
\newblock In \emph{Proceedings of the 1988 connectionist models summer school},
  volume~1, pages 21--28. CMU, Pittsburgh, Pa: Morgan Kaufmann, 1988.

\bibitem[Li et~al.(2018)Li, Chen, Tai, and E]{LCTE2018}
Q.~Li, L.~Chen, C.~Tai, and W.~E.
\newblock Maximum principle based algorithms for deep learning.
\newblock \emph{The Journal of Machine Learning Research}, 18\penalty0
  (1):\penalty0 5998--6026, 2018.

\bibitem[Li and Hoiem(2017)]{LH2017}
Z.~Li and D.~Hoiem.
\newblock Learning without forgetting.
\newblock \emph{IEEE transactions on pattern analysis and machine
  intelligence}, 40\penalty0 (12):\penalty0 2935--2947, 2017.

\bibitem[Lions et~al.(2001)Lions, Maday, and Turinici]{LMT2001}
J.~Lions, Y.~Maday, and G.~Turinici.
\newblock {R{\'e}solution d'{EDP} par un sch{\'e}ma en temps parar{\'e}el}.
\newblock \emph{C. R. Acad. Sci. Paris}, 2001.
\newblock t. 332, S{\'e}rie I, p. 661-668.

\bibitem[Lu et~al.(2018)Lu, Zhong, Li, and Dong]{LZLD2018}
Y.~Lu, A.~Zhong, Q.~Li, and B.~Dong.
\newblock Beyond finite layer neural networks: Bridging deep architectures and
  numerical differential equations.
\newblock In J.~Dy and A.~Krause, editors, \emph{Proceedings of the 35th
  International Conference on Machine Learning}, volume~80 of \emph{Proceedings
  of Machine Learning Research}, pages 3282--3291, Stockholmsmässan, Stockholm
  Sweden, 10--15 Jul 2018. PMLR.
\newblock URL \url{http://proceedings.mlr.press/v80/lu18d.html}.

\bibitem[Maday and Mula(2020)]{MM2020}
Y.~Maday and O.~Mula.
\newblock An adaptive parareal algorithm.
\newblock \emph{To appear in the Journal of Computational and Applied
  Mathematics}, 2020.
\newblock URL \url{https://arxiv.org/abs/1909.08333}.

\bibitem[Pineda(1988)]{Pineda1988}
F.~J. Pineda.
\newblock Generalization of back propagation to recurrent and higher order
  neural networks.
\newblock In \emph{Neural information processing systems}, pages 602--611,
  1988.

\bibitem[Pinelis and Sakhanenko(1986)]{PS1986}
I.~F. Pinelis and A.~I. Sakhanenko.
\newblock Remarks on inequalities for large deviation probabilities.
\newblock \emph{Theory of Probability \& Its Applications}, 30\penalty0
  (1):\penalty0 143--148, 1986.

\bibitem[Pontryagin(2018)]{Pontryagin2018}
L.~S. Pontryagin.
\newblock \emph{Mathematical theory of optimal processes}.
\newblock Routledge, 2018.

\bibitem[Robbins and Monro(1951)]{RM1951}
H.~Robbins and S.~Monro.
\newblock A stochastic approximation method.
\newblock \emph{The annals of mathematical statistics}, pages 400--407, 1951.

\bibitem[Rosenblatt(1958)]{Rosenblatt1958}
F.~Rosenblatt.
\newblock The perceptron: A probabilistic model for information storage and
  organization in the brain.
\newblock \emph{Psychological Review}, pages 65--386, 1958.

\bibitem[Silver et~al.(2016)Silver, Huang, Maddison, Guez, Sifre, Van
  Den~Driessche, Schrittwieser, Antonoglou, Panneershelvam, Lanctot,
  et~al.]{GO-game-2016}
D.~Silver, A.~Huang, C.~J. Maddison, A.~Guez, L.~Sifre, G.~Van Den~Driessche,
  J.~Schrittwieser, I.~Antonoglou, V.~Panneershelvam, M.~Lanctot, et~al.
\newblock Mastering the game of go with deep neural networks and tree search.
\newblock \emph{nature}, 529\penalty0 (7587):\penalty0 484, 2016.

\bibitem[Telgarsky(2015)]{Telgarsky2015}
M.~Telgarsky.
\newblock Representation benefits of deep feedforward networks.
\newblock \emph{arXiv preprint arXiv:1509.08101}, 2015.

\bibitem[Vialard et~al.(2020)Vialard, Kwitt, Wei, and Niethammer]{VKWN2020}
F.~Vialard, R.~Kwitt, S.~Wei, and M.~Niethammer.
\newblock A shooting formulation of deep learning, 2020.

\bibitem[Wei et~al.(2016)Wei, Wang, Rui, and Chen]{WWRC2016}
T.~Wei, C.~Wang, Y.~Rui, and C.~W. Chen.
\newblock Network morphism.
\newblock In \emph{International Conference on Machine Learning}, pages
  564--572, 2016.

\bibitem[Weinan(2017)]{E2017}
E.~Weinan.
\newblock A proposal on machine learning via dynamical systems.
\newblock \emph{Communications in Mathematics and Statistics}, 5\penalty0
  (1):\penalty0 1--11, 2017.
\newblock \doi{10.1007/s40304-017-0103-z}.

\bibitem[Wu et~al.(2016)Wu, Schuster, Chen, Le, Norouzi, Macherey, Krikun, Cao,
  Gao, Macherey, et~al.]{WSCLNMKCG2016}
Y.~Wu, M.~Schuster, Z.~Chen, Q.~V. Le, M.~Norouzi, W.~Macherey, M.~Krikun,
  Y.~Cao, Q.~Gao, K.~Macherey, et~al.
\newblock Google's neural machine translation system: Bridging the gap between
  human and machine translation.
\newblock \emph{arXiv preprint arXiv:1609.08144}, 2016.

\bibitem[Yarotsky(2017)]{Yarotsky2017}
D.~Yarotsky.
\newblock Error bounds for approximations with deep {ReLU} networks.
\newblock \emph{Neural Networks}, 94:\penalty0 103--114, 2017.

\bibitem[Zeiler(2012)]{Zeiler2012}
M.~D. Zeiler.
\newblock Adadelta: an adaptive learning rate method.
\newblock \emph{arXiv preprint arXiv:1212.5701}, 2012.

\bibitem[Zhang et~al.(2017)Zhang, Li, Change~Loy, and Lin]{ZLCL2017}
X.~Zhang, Z.~Li, C.~Change~Loy, and D.~Lin.
\newblock Polynet: A pursuit of structural diversity in very deep networks.
\newblock In \emph{Proceedings of the IEEE Conference on Computer Vision and
  Pattern Recognition}, pages 718--726, 2017.

\end{thebibliography}

\appendix
\section{Proof of Lemma \ref{lem:diffLoss}}
\label{app:proof-prop}
The proof of Lemma \ref{lem:diffLoss} makes use of certain bounds which we gather in the following Proposition \ref{prop:gronwall-evol}. They are derived by an immediate extension of the proofs given in \cite{LCTE2018} so we just sketch the proof for self-completness of the current work.

\begin{proposition}
\label{prop:gronwall-evol}
We have the following bounds for all $t\in [0,T]$,
\begin{align}
\Vert p^{\btheta, \zeta}_t \Vert &\leq K' \coloneqq K e^{\Vert \Pi_\zeta \Vert K T}  \label{eq:Kprime}\\
\Vert \delta u_t \Vert 
&\leq
K_{\delta u} \coloneqq
\left( \sqrt{T} (\eta + \zeta) + \int_0^T \Vert f(s, u^{\btheta, \zeta}_s, \varphi_s) -f(s, u^{\btheta, \zeta}_s, \theta_s) \Vert \ds \right)e^{KT}\label{eq:Kz}\\
\Vert \delta p_t \Vert &\leq K_{\delta p} \coloneqq \sqrt{T}e^{2KT} (1+K(1+K'T)) (\eta+\zeta) \\
&\qquad+ e^{2KT}K(1+K'T)  \int_0^T \Vert f(s, u^{\btheta, \zeta}_s, \varphi_s) -f(s, u^{\btheta, \zeta}_s, \theta_s) \Vert \ds \\
&\qquad+ e^{KT}  \int_0^T \Vert  \nabla_u H(s, u^{\btheta, \zeta}_s, p^{\btheta, \zeta}_s, \varphi_s) -  \nabla_u H(s, u^{\btheta, \eta}_s, p^{\btheta, \eta}_s, \theta_s)\Vert \ds \label{eq:Kp}
\end{align}
\end{proposition}

\begin{proof}[Proof of Proposition \ref{prop:gronwall-evol}]
Inequality \eqref{eq:Kprime} follows from the fact that
$$
\dot p^{\btheta, \zeta}_t
= \Pi_\zeta \left( - p^{\btheta, \zeta}_t \cdot \nabla_u f( u^{\btheta, \zeta}_t, \theta_t)  \right)
$$
thus taking the scalar product with  $p^{\btheta, \zeta}_t$ and using the Cauchy-Schwarz inequality, we derive the evolution
$$
\frac 1 2 \frac{\dd}{\dt} \Vert p^{\btheta, \zeta}_t \Vert^2 \leq \Vert \Pi_\zeta \Vert \Vert \nabla_u f( u^{\btheta, \zeta}_t, \theta_t)  \Vert_2 \Vert p^{\btheta, \zeta}_t \Vert^2,
\quad
\Vert p^{\btheta, \zeta}_T \Vert^2 = \Vert \nabla_u \Phi(u^{\btheta, \zeta}_T, y) \Vert^2
$$
and the result follows by applying the Grönwall's inequality and the fact that $\Vert \nabla_u \Phi(u^{\btheta, \zeta}_T, y) \Vert \leq K$ and $\Vert \nabla_u f( u^{\btheta, \zeta}_t, \theta_t) \Vert_2 \leq K$ by hypothesis (A1) and (A2).

To derive inequality \eqref{eq:Kz}, we start from
$$
\delta \dot u_t
= \dot u_t^{\bphi, \eta} - \dot u_t^{\btheta, \zeta}
= \Pi_\eta f( u^{\bphi, \eta}_t, \varphi_t) - \Pi_\zeta f( u^{\btheta, \zeta}_t, \theta_t).
$$
Integrating between $0$ and $t\leq T$,
$$
\delta u_t
=
\int_0^t \Pi_\eta f(s, u^{\bphi, \eta}_s, \varphi_s) - \Pi_\zeta f(s, u^{\btheta, \zeta}_s, \theta_s) \ds
$$
Taking norms,
\begin{align}
\Vert \delta u_t \Vert
&\leq 
\int_0^t \Vert f(s, u^{\bphi, \eta}_s, \varphi_s)  - \Pi_\eta f(s, u^{\bphi, \eta}_s, \varphi_s) \Vert \ds
+
\int_0^t \Vert f(s, u^{\btheta, \zeta}_s, \theta_s) - \Pi_\zeta f(s, u^{\btheta, \zeta}_s, \theta_s) \Vert \ds \\
&+
\int_0^t \Vert f(s, u^{\bphi, \eta}_s, \varphi_s) - f(s, u^{\btheta, \zeta}_s, \theta_s) \Vert \ds \\
&\leq
\sqrt{T} (\eta + \zeta)
+
\int_0^T \Vert f(s, u^{\bphi, \eta}_s, \varphi_s) - f(s, u^{\btheta, \zeta}_s, \varphi_s) \Vert \ds
+
\int_0^T \Vert f(s, u^{\btheta, \zeta}_s, \varphi_s) -f(s, u^{\btheta, \zeta}_s, \theta_s) \Vert \ds
\\
&\leq
\sqrt{T} (\eta + \zeta) + K \int_0^T \Vert \delta u_s \Vert \ds
+
\int_0^T \Vert f(s, u^{\btheta, \zeta}_s, \varphi_s) -f(s, u^{\btheta, \zeta}_s, \theta_s) \Vert \ds
\end{align}
Thus by Grönwall's inequality,
$$
\Vert \delta u_t \Vert 
\leq
K_{\delta u} \coloneqq
\left( \sqrt{T} (\eta + \zeta) + \int_0^T \Vert f(s, u^{\btheta, \zeta}_s, \varphi_s) -f(s, u^{\btheta, \zeta}_s, \theta_s) \Vert \ds \right)e^{KT}
$$
Proceeding in a similar manner with $\delta \dot p_t$, by integrating between $T$ and $t$, we derive
\begin{align}
\int_T^t \delta \dot p_s \ds
&=
\delta p_t - \delta p_T
= \int_T^t \Pi_\eta \nabla_u H(s, u^{\bphi, \eta}_s, p^{\bphi, \eta}_s, \varphi_s) -\Pi_\zeta \nabla_u H(s, u^{\btheta, \zeta}_s, p^{\btheta, \zeta}_s, \theta_s) \ds
\end{align}
and similarly as before, we infer that
\begin{align}
\Vert \delta p_t \Vert
&\leq
 \sqrt{T}(\eta+\zeta) + K \Vert \delta u_T \Vert + K K' \int_0^T \Vert \delta u_s \Vert + K \int_0^T \Vert \delta p_s \Vert \ds
 \ds \\
&+ \int_0^T \Vert  \nabla_u H(s, u^{\btheta, \zeta}_s, p^{\btheta, \zeta}_s, \varphi_s) -  \nabla_u H(s, u^{\btheta, \zeta}_s, p^{\btheta, \zeta}_s, \theta_s)\Vert \ds
\end{align}
Using bounds \eqref{eq:Kz} and \eqref{eq:Kp}, and applying the Gronwall inequality, we derive the announced estimate
\begin{align}
\Vert \delta p_t \Vert  \leq K_{\delta p}
\end{align}
with
\begin{align}
K_{\delta p}&= e^{KT}
\left(
 \sqrt{T}(\eta+\zeta) + K K_{\delta u}(1 + K' T) + \int_0^T \Vert  \nabla_u H(s, u^{\btheta, \zeta}_s, p^{\btheta, \zeta}_s, \varphi_s) -  \nabla_u H(s, u^{\btheta, \zeta}_s, p^{\btheta, \zeta}_s, \theta_s)\Vert \ds
\right) \\
&= \sqrt{T}e^{2KT} (1+K(1+K'T)) (\eta+\zeta) + e^{2KT}K(1+K'T)  \int_0^T \Vert f(s, u^{\btheta, \zeta}_s, \varphi_s) -f(s, u^{\btheta, \zeta}_s, \theta_s) \Vert \ds \\
&+ e^{KT}  \int_0^T \Vert  \nabla_u H(s, u^{\btheta, \zeta}_s, p^{\btheta, \zeta}_s, \varphi_s) -  \nabla_u H(s, u^{\btheta, \zeta}_s, p^{\btheta, \zeta}_s, \theta_s)\Vert \ds
\end{align}
\end{proof}

We can now prove Lemma \ref{lem:diffLoss}.

\begin{proof}[Proof of Lemma \ref{lem:diffLoss}]
By definition \eqref{eq:H} of the Hamiltonian, for any $\btheta \in \st$,
$$
I(\bu^\btheta, \bp^\btheta, \btheta) \coloneqq \int_0^T \left( p_t^\btheta \cdot f( u_t^\btheta, \theta_t) - H(t, u_t^\btheta, p_t^\btheta, \theta_t) - L(\theta_t) \right)\dt = 0.
$$
Thus
\begin{align}
\label{eq:diff-I}
0
&= I(\bu^{\bphi, \eta}, \bp^{\bphi, \eta}, \bphi) - I(\bu^{\btheta, \zeta}, \bp^{\btheta, \zeta}, \btheta) \\
&=
\underbrace{\int_0^T p_t^{\bphi, \eta} \cdot f( u_t^{\bphi, \eta}, \varphi_t) - p_t^{\btheta, \zeta} \cdot f( u_t^{\btheta, \zeta}, \theta_t) \dt}_{\coloneqq \cI_1}
-
\underbrace{\int_0^T H(t, u_t^{\bphi, \eta}, p_t^{\bphi, \eta}, \varphi_t)-H(t, u_t^{\btheta, \zeta}, p_t^{\btheta, \zeta}, \theta_t) \dt}_{\coloneqq \cI_2} \\
&-
\underbrace{\int_0^T R(\varphi_t) - R(\theta_t) \dt}_{\coloneqq \cI_3}
\end{align}
Denoting
$$
e^{f, \bphi, \eta}_t \coloneqq f( u_t^{\bphi, \eta}, \varphi_t) - \Pi_\eta f( u_t^{\bphi, \eta}, \varphi_t)
$$
and similarly for $e^{f, \btheta, \zeta}_t$, we have
\begin{align}
\cI_1 =
\underbrace{\int_0^T p_t^{\bphi, \eta} \cdot \Pi_\eta f( u_t^{\bphi, \eta}, \varphi_t) - p_t^{\btheta, \zeta} \cdot \Pi_\zeta f( u_t^{\btheta, \zeta}, \theta_t) \dt}_{\coloneqq \cI_{1,1}}
+
\underbrace{\int_0^T p_t^{\bphi, \eta} \cdot e^{f, \bphi, \eta}_t - p_t^{\btheta, \zeta} \cdot e^{f, \btheta, \zeta}_t \dt}_{\coloneqq \cI_{1,2}}
\end{align}
In addition, since $\dot u_t^{\bphi, \eta} =\Pi_\eta f( u_t^{\bphi, \eta}, \varphi_t)$ and similarly for $\btheta$, it follows that
\begin{align}
\cI_{1,1} = \int_0^T p_t^{\bphi, \eta} \cdot \dot u_t^{\bphi, \eta} - p_t^{\btheta, \zeta} \cdot  \dot u_t^{\btheta, \zeta} \dt
=
\underbrace{\int_0^T p_t^{\btheta, \zeta} \cdot \delta \dot u_t + \delta p_t \cdot \dot u_t^{\btheta, \zeta} \dt}_{\coloneqq \cI_{1,1,1}}
+
\underbrace{\int_0^T \delta p_t \cdot \delta \dot u_t \dt}_{\coloneqq \cI_{1,1,2}}
\end{align}
where
$$
\delta u_t \coloneqq u_t^{\bphi,\eta} - u_t^{\btheta,\zeta},
\quad \text{and} \quad
\delta p_t \coloneqq p_t^{\bphi,\eta} - p_t^{\btheta,\zeta}.
$$
By integration by parts, and the fact that $\dot p_t^{\btheta, \zeta} = - \Pi_\zeta \nabla_u H(t, u_t^{\btheta, \zeta}, p_t^{\btheta, \zeta}, \theta_t)$, we have
\begin{align}
\cI_{1,1,1}
&= p_t^{\btheta, \zeta} \cdot \delta u_t \, \Big\vert_0^T
+
\int_0^T \delta p_t \cdot \dot u_t^{\btheta, \zeta} - \dot p_t^{\btheta, \zeta}\cdot \delta u_t \dt \\
&= p_t^{\btheta, \zeta} \cdot \delta u_t \, \Big\vert_0^T
+ \int_0^T 
\delta u_t \cdot \Pi_\zeta \nabla_u H(t, u_t^{\btheta, \zeta}, p_t^{\btheta, \zeta}, \theta_t)
+ \delta p_t \cdot \Pi_\zeta \nabla_p H(t, u_t^{\btheta, \zeta}, p_t^{\btheta, \zeta}, \theta_t) \dt  \\
&= p_t^{\btheta, \zeta} \cdot \delta u_t \, \Big\vert_0^T + \int_0^T \Pi_\zeta \nabla_{w} H(t, w_t^{\btheta, \zeta}, \theta_t) \cdot \delta w_t\dt,
\end{align}
where we have used the shorthand notation $w = (u, p)$ in the last line. This notation will also be used in what follows.

We next address the intergral $\cI_{1,1,2}$. By integration by parts,
\begin{align}
\cI_{1,1,2}
&= \frac 1 2 \int_0^T \delta p_t \cdot \delta \dot u_t\dt + \frac 1 2 \int_0^T \delta p_t \cdot \delta \dot u_t\dt \\
&= \frac 1 2 \delta p_t \cdot \delta u_t \Big\vert_0^T \
- \frac 1 2 \int_0^T \delta \dot p_t \cdot \delta u_t\dt
+ \frac 1 2 \int_0^T \delta p_t \cdot \delta \dot u_t\dt \\
&= \frac 1 2 \delta p_t \cdot \delta u_t \Big\vert_0^T
+ \frac 1 2 \int_0^T \left( \Pi_\eta \nabla_{w} H(t, w_t^{\bphi, \eta}, \varphi_t) - \Pi_\zeta \nabla_{w} H(t, w_t^{\btheta, \zeta}, \theta_t) \right)
\cdot \delta w_t \dt
\end{align}
Adding and subtracting $\nabla_{w} H(t, w_t^{\bphi, \eta}, \varphi_t)$ and $\nabla_{w} H(t, w_t^{\btheta, \zeta}, \theta_t)$ to the last formula,
\begin{align}
\cI_{1,1,2}
&= \frac 1 2 \delta p_t \cdot \delta u_t \Big\vert_0^T
+ \frac 1 2 \int_0^T \left( 
\Pi_\eta \nabla_{w} H(t, w_t^{\bphi, \eta}, \varphi_t) - \nabla_{w} H(t, w_t^{\bphi, \eta}, \varphi_t)
\right) \cdot \delta w_t \dt \\
&+
\frac 1 2 \int_0^T
\left(
\nabla_{w} H(t, w_t^{\btheta, \zeta}, \theta_t) 
-
\Pi_\zeta \nabla_{w} H(t, w_t^{\btheta, \zeta}, \theta_t) 
\right)\cdot \delta w_t \dt  \\
&+ \frac 1 2 \int_0^T
\left(
\nabla_{w} H(t, w_t^{\bphi, \eta}, \varphi_t)
-
\nabla_{w} H(t, w_t^{\btheta, \zeta}, \theta_t) 
\right) \cdot \delta w_t \dt
\end{align}
By applying Taylor's theorem around $w=w^{\btheta, \zeta}= (u_t^{\btheta, \zeta}, p_t^{\btheta, \zeta})$ to the function $w \to \nabla_{w}H(t, w, \varphi_t)$, there exists $r_1(t)\in [0,1]$ such that
$$
\nabla_{w} H(t, w_t^{\bphi, \eta}, \varphi_t)
=
\nabla_{w} H(t, w_t^{\btheta, \zeta}, \varphi_t)
+ \delta w_t \cdot
 \nabla_{w}^2 H(t, w_t^{\btheta, \zeta}+r_1(t) \delta w_t, \varphi_t)
$$
Thus
\begin{align}
\cI_{1,1,2}
&= \frac 1 2 \delta p_t \cdot \delta u_t \Big\vert_0^T
+ \frac 1 2 \int_0^T \left( 
\Pi_\eta \nabla_{w} H(t, w_t^{\bphi, \eta}, \varphi_t) - \nabla_{w} H(t, w_t^{\bphi, \eta}, \varphi_t)
\right) \cdot \delta w_t \dt \\
&+
\frac 1 2 \int_0^T
\left(
\nabla_{w} H(t, w_t^{\btheta, \zeta}, \theta_t) 
-
\Pi_\zeta \nabla_{w} H(t, w_t^{\btheta, \zeta}, \theta_t) 
\right)\cdot \delta w_t \dt  \\
&+ \frac 1 2 \int_0^T
\left(
\nabla_{w} H(t, w_t^{\btheta, \zeta}, \varphi_t)
-
\nabla_{w} H(t, w_t^{\btheta, \zeta}, \theta_t) 
\right) \cdot \delta w_t \dt \\
&+ \frac 1 2 \int_0^T \delta w_t \cdot
 \nabla_{w}^2 H(t, w_t^{\btheta, \zeta}+r_1(t) \delta w_t, \varphi_t)
 \cdot \delta w_t \dt
\end{align}
Since $\delta u_0=0$, the boundary terms from $\cI_{1,1,1}$ and $\cI_{1,1,2}$ are
\begin{align}
 \left( p_t^{\btheta, \zeta} + \frac 1 2  \delta p_t \right) \cdot\delta u_t \Big\vert_0^T
& =  \left( p_T^{\btheta, \zeta} + \frac 1 2 \delta p_T \right) \cdot\delta u_T \\
&= -\nabla \Phi(u_T^{\btheta, \zeta})\cdot\delta u_T - \frac 1 2 \left( \nabla \Phi(u_T^{\bphi, \eta}) - \nabla \Phi(u_T^{\btheta, \zeta})  \right) \cdot \delta u_T
\end{align}
By applying Taylor's theorem around $z=u_t^{\btheta, \zeta}$ to the function $z\to \nabla \Phi(z)$, there exists $r_2\in [0,1]$ such that
$$
\nabla \Phi(u_T^{\bphi, \eta}) = \nabla \Phi(u_T^{\btheta, \zeta}) + \delta u_t \cdot \nabla^2 \Phi( u_T^{\btheta, \zeta} + r_2 \delta u_T) .
$$
Therefore
$$
 \left( p_t^{\btheta, \zeta} +  \frac 1 2  \delta p_t \right) \cdot\delta u_t \Big\vert_0^T
=
-\nabla \Phi(u_T^{\btheta, \zeta})\cdot\delta u_T - \frac 1 2 \delta u_T \cdot \nabla^2 \Phi( u_T^{\btheta, \zeta} + r_2 \delta u_T) \cdot \delta u_T
$$
Since, again by Taylor's theorem, there exists $r_3\in [0,1]$ such that
$$
\Phi(p_T^{\bphi,\eta})
=
\Phi(p_T^{\btheta,\zeta})  + \delta u_t \cdot \nabla \Phi (u_T^{\btheta, \zeta}) + \frac 1 2 \delta u_t \cdot \nabla^2 \Phi (u_T^{\btheta, \zeta} + r_3 \delta u_t) \cdot \delta u_T,
$$
we finally get the expression for the the boundary terms from $\cI_{1,1,1}$ and $\cI_{1,1,2}$
$$
\left( p_t^{\btheta, \zeta} + \frac 1 2   \delta p_t \right) \cdot\delta u_t \Big\vert_0^T
=
\Phi(u_T^{\btheta, \zeta}) - \Phi(u_T^{\bphi, \eta}) - \frac 1 2 \delta u_T \cdot
\left( 
\nabla^2\Phi(u_T^{\btheta, \zeta} + r_2 \delta u_T) - \nabla^2\Phi(u_T^{\btheta, \zeta} + r_3 \delta u_T^\eta) 
\right) \cdot \delta u_T
$$
We now turn to $\cI_2$. By Taylor's theorem, there exists $r_4(t) \in [0,1]$ such that
\begin{align}
\cI_2
&= \int_0^T H(t, w_t^{\bphi, \eta}, \varphi_t)-H(t, w_t^{\btheta, \zeta}, \theta_t) \dt \\
&= \int_0^T \Delta H_{\bphi, \btheta}^{x, \zeta}(t) \dt
+ \int_0^T \nabla_w H(t, w_t^{\btheta, \zeta}, \varphi_t) \cdot \delta w_t \dt \\
&\quad + \frac 1 2 \int_0^T \delta w_t \cdot \nabla_w^2 H(t, w_t^{\btheta, \zeta} + r_4(t) \delta w_t, \varphi_t) \cdot \delta w_t \dt,
\end{align}
where we remind that, as defined in equation \eqref{eq:deltaH},
$
\Delta H_{\bphi, \btheta}^{x, \zeta}(t) \coloneqq H(t, w_t^{\btheta, \zeta}, \varphi_t)-H(t, w_t^{\btheta, \zeta}, \theta_t).
$
Injecting the above expressions for the terms of $\cI_1$ and $\cI_2$ into equation \eqref{eq:diff-I}, passing the terms corresponding to the loss function to the left hand side, and adding and subtracting $\int_0^T \Pi_\zeta\nabla_{w} H(t, w_t^{\btheta, \zeta}, \varphi_t)\cdot \delta w_t\dt$,
\begingroup
\allowdisplaybreaks
%\begin{align}
%&\loss(x, y, \bphi) -  \loss(x, y, \btheta)  \\
%&=\Phi(u_T^{\bphi, \eta}) - \Phi(u_T^{\btheta, \zeta})  + \int_0^T \left( R(\varphi_t) - R(\theta_t) \right) \dt \\
%&=  \int_0^T \Pi_\zeta \nabla_{w} H(t, w_t^{\btheta, \zeta}, \theta_t) \cdot \delta w_t\dt + \int_0^T p_t^{\bphi, \eta} \cdot e^{f, \bphi, \eta}_t - p_t^{\btheta, \zeta} \cdot e^{f, \btheta, \zeta}_t \dt \\
%&+ \frac 1 2 \int_0^T \left( 
%\Pi_\eta \nabla_{w} H(t, w_t^{\bphi, \eta}, \varphi_t) - \nabla_{w} H(t, w_t^{\bphi, \eta}, \varphi_t)
%\right) \cdot \delta w_t \dt \\
%&+
%\frac 1 2 \int_0^T
%\left(
%\nabla_{w} H(t, w_t^{\btheta, \zeta}, \theta_t) 
%-
%\Pi_\zeta \nabla_{w} H(t, w_t^{\btheta, \zeta}, \theta_t) 
%\right)\cdot \delta w_t \dt  \\
%&+ \frac 1 2 \int_0^T
%\left(
%\nabla_{w} H(t, w_t^{\btheta, \zeta}, \varphi_t)
%-
%\nabla_{w} H(t, w_t^{\btheta, \zeta}, \theta_t) 
%\right) \cdot \delta w_t \dt \\
%&+ \frac 1 2 \int_0^T \delta w_t \cdot
% \nabla_{w}^2 H(t, w_t^{\btheta, \zeta}+r_1(t) \delta w_t, \varphi_t)
% \cdot \delta w_t \dt  \\
% & - \int_0^T \Delta H_{\bphi, \btheta}^{x, \zeta}(t) \dt
%- \int_0^T \nabla_w H(t, w_t^{\btheta, \zeta}, \varphi_t) \cdot \delta w_t \dt \\
%& - \frac 1 2 \int_0^T \delta w_t \cdot \nabla_w^2 H(t, w_t^{\btheta, \zeta} + r_4(t) \delta w_t, \varphi_t) \cdot \delta w_t \dt \\
%&- \frac 1 2 \delta u_T \cdot
%\left( \nabla^2\Phi(u_T^{\btheta, \zeta} + r_2 \delta u_T) - \nabla^2\Phi(u_T^{\btheta, \zeta} + r_3 \delta u_T^\eta) \right) 
%\end{align}
%
\begin{align}
&\loss(x, y, \bphi) -  \loss(x, y, \btheta) \tag*{(Diff-Loss)} \\
&=\Phi(u_T^{\bphi, \eta}) - \Phi(u_T^{\btheta, \zeta})  + \int_0^T \left( R(\varphi_t) - R(\theta_t) \right) \dt \\
&=  \int_0^T \left( \Pi_\zeta \nabla_{w} H(t, w_t^{\btheta, \zeta}, \theta_t)- \Pi_\zeta\nabla_{w} H(t, w_t^{\btheta, \zeta}, \varphi_t)\right) \cdot \delta w_t\dt \tag*{(T1)}\\
&+ \int_0^T p_t^{\bphi, \eta} \cdot e^{f, \bphi, \eta}_t - p_t^{\btheta, \zeta} \cdot e^{f, \btheta, \zeta}_t \dt \tag*{(T2)} \\
&+ \frac 1 2 \int_0^T \left( 
\Pi_\eta \nabla_{w} H(t, w_t^{\bphi, \eta}, \varphi_t) - \nabla_{w} H(t, w_t^{\bphi, \eta}, \varphi_t)
\right) \cdot \delta w_t \dt \tag*{(T3)}\\
&+
\frac 1 2 \int_0^T
\left(
\nabla_{w} H(t, w_t^{\btheta, \zeta}, \theta_t) 
-
\Pi_\zeta \nabla_{w} H(t, w_t^{\btheta, \zeta}, \theta_t) 
\right)\cdot \delta w_t \dt  \tag*{(T4)}\\
&+ \frac 1 2 \int_0^T
\left(
\nabla_{w} H(t, w_t^{\btheta, \zeta}, \varphi_t)
-
\nabla_{w} H(t, w_t^{\btheta, \zeta}, \theta_t) 
\right) \cdot \delta w_t \dt \tag*{(T5)}\\
&+ \frac 1 2 \int_0^T \delta w_t \cdot
 \nabla_{w}^2 H(t, w_t^{\btheta, \zeta}+r_1(t) \delta w_t, \varphi_t)
 \cdot \delta w_t \dt  \tag*{(T6)}\\
 & - \int_0^T \Delta H_{\bphi, \btheta}^{x, \zeta}(t) \dt  \tag*{(T7)}\\
 &+ \int_0^T  \left( \Pi_\zeta\nabla_{w} H(t, w_t^{\btheta, \zeta}, \varphi_t) - \nabla_w H(t, w_t^{\btheta, \zeta}, \varphi_t) \right) \cdot \delta w_t \dt \tag*{(T8)}\\
& - \frac 1 2 \int_0^T \delta w_t \cdot \nabla_w^2 H(t, w_t^{\btheta, \zeta} + r_4(t) \delta w_t, \varphi_t) \cdot \delta w_t \dt \tag*{(T9)}\\
&- \frac 1 2 \delta u_T \cdot
\left( \nabla^2\Phi(u_T^{\btheta, \zeta} + r_2 \delta u_T) - \nabla^2\Phi(u_T^{\btheta, \zeta} + r_3 \delta u_T^\eta) \right)  \tag{(T10)}
\end{align}
\endgroup
We next derive a bound for the difference $\loss(x, y, \bphi) -  \loss(x, y, \btheta)$. We proceed to bound all the terms by order of appearance in the above formula. We will sometimes use that by \eqref{eq:Kz}, \eqref{eq:Kp} and Jensen's inequality,
\begin{align}
\Vert \delta w_t  \Vert^2
=
\Vert \delta u_t  \Vert^2 + \Vert \delta p_t  \Vert^2
 \lesssim 
(\eta+\zeta)^2 + 
 \int_0^T \Vert  \nabla_w H(s, w^{\btheta, \zeta}_s, \varphi_s) -  \nabla_w H(s, w^{\btheta, \zeta}_s,  \theta_s)\Vert^2 \ds 
\end{align}
which yields
\begin{equation}
\Vert \delta w_t  \Vert_{L^2([0,T])} 
\lesssim
(\eta + \zeta) + \Vert \nabla_w (\Delta H_{\bphi, \btheta}^{x, \zeta} )\Vert_{L^2([0,T], \bR^n\times \bR^n)}.
\label{eq:delta-w-bound}
\end{equation}

\paragraph*{Term (T1)+(T5):} By the Cauchy-Schwartz inequality and \eqref{eq:delta-w-bound}, we can bound term (T1) by
\begin{align}
&\int_0^T \left( \Pi_\zeta \nabla_{w} H(t, w_t^{\btheta, \zeta}, \theta_t)- \Pi_\zeta\nabla_{w} H(t, w_t^{\btheta, \zeta}, \varphi_t)\right) \cdot \delta w_t\dt  \\
&\leq
\Vert \Pi_\zeta \Vert
\Vert \nabla_{w} H(t, w_t^{\btheta, \zeta}, \theta_t)- \nabla_{w} H(t, w_t^{\btheta, \zeta}, \varphi_t) \Vert_{L^2([0,T])} \Vert \delta w_t \Vert_{L^2([0,T])} \\
&\lesssim
(\eta + \zeta) \Vert \nabla_w ( \Delta H_{\bphi, \btheta}^{x, \zeta} ) \Vert_{L^2([0,T])} 
+ \Vert \nabla_w (\Delta H_{\bphi, \btheta}^{x, \zeta} )\Vert_{L^2([0,T], \bR^n\times \bR^n)}^2
\end{align}
Thus
$$
(T1) + (T5) \lesssim (\eta + \zeta) + (\eta + \zeta) \Vert \nabla_w ( \Delta H_{\bphi, \btheta}^{x, \zeta} ) \Vert_{L^2([0,T])} 
+ \Vert \nabla_w (\Delta H_{\bphi, \btheta}^{x, \zeta} )\Vert_{L^2([0,T])}^2
$$

\paragraph*{Term (T2):} Using \eqref{eq:Kprime} and the fact that $\Vert e^{f, \btheta, \zeta}_t  \Vert_{L^2([0,T], \bR^n)} \leq \zeta$
\begin{align}
\int_0^T p_t^{\bphi, \eta} \cdot e^{f, \bphi, \eta}_t - p_t^{\btheta, \zeta} \cdot e^{f, \btheta, \zeta}_t \dt
\lesssim \eta+\zeta
\end{align}

\paragraph*{Terms (T3) and (T8):} By the Cauchy-Schwarz inequality, we can bound term (T3)
\begin{align}
\vert
&\int_0^T \left( 
\Pi_\eta \nabla_{w} H(t, w_t^{\bphi, \eta}, \varphi_t) - \nabla_{w} H(t, w_t^{\bphi, \eta}, \varphi_t)
\right) \cdot \delta w_t \dt \vert  \\
&\leq \Vert \Pi_\eta \nabla_{w} H(t, w_t^{\bphi, \eta}, \varphi_t) - \nabla_{w} H(t, w_t^{\bphi, \eta}, \varphi_t) \Vert_{L^2([0,T], \bR^n\times \bR^n)}
\Vert \delta w_t  \Vert_{L^2([0,T], \bR^n\times \bR^n)}  \\
&\lesssim \eta \Vert \nabla_w (\Delta H_{\bphi, \btheta}^{x, \zeta} )\Vert_{L^2([0,T], \bR^n\times \bR^n)}
\end{align}
where we have used that the propagator $\Pi_\eta$ guarantees
$$
\Vert \Pi_\eta \nabla_{w} H(t, w_t^{\bphi, \eta}, \varphi_t) - \nabla_{w} H(t, w_t^{\bphi, \eta}, \varphi_t) \Vert_{L^2([0,T], \bR^n\times \bR^n)},
\lesssim \eta
$$
and finally
\begin{align}
&\vert
\int_0^T \left(
\Pi_\eta \nabla_{w} H(t, w_t^{\bphi, \eta}, \varphi_t) - \nabla_{w} H(t, w_t^{\bphi, \eta}, \varphi_t)
\right) \cdot \delta w_t \dt \vert \lesssim
\eta(\eta + \zeta) + \eta \Vert \nabla_w (\Delta H_{\bphi, \btheta}^{x, \zeta} )\Vert_{L^2([0,T], \bR^n\times \bR^n)}
\end{align}

\paragraph*{Term (T4):}
Using \eqref{eq:Kprime} and the fact that $\Vert e^{f, \btheta, \zeta}_t  \Vert_{L^2([0,T], \bR^n)} \leq \zeta$, we have $\int_0^T p_t^{\bphi, \eta} \cdot e^{f, \bphi, \eta}_t - p_t^{\btheta, \zeta} \cdot e^{f, \btheta, \zeta}_t \dt \lesssim \eta+\zeta$.

\paragraph*{Term (T6), (T9) and (T10):} By assumptions (A1), (A2), all second derivative terms are bounded element-wise by some constant $K$. Hence, we have $| \delta w_t \cdot A \cdot \delta w_t  | \lesssim \Vert \delta w_t  \Vert^2$ for $A$ being a second derivative matrix.

\paragraph*{Summary:} By using the above inequalities, in formula (Diff-Loss), we deduce that there exists a constant $C>0$ depending on $T$, $\Vert \Pi_\zeta\Vert$ and $\Vert \Pi_\eta\Vert$ such that, if $\eta\leq 1$,
\begin{align}
&\loss(x, y, \bphi) -  \loss(x, y, \btheta)  \leq - \int_0^T \Delta H_{\bphi, \btheta}^{x, \zeta}(t)\dt  + C \left( (\eta + \zeta)^2  + \Vert \nabla_w (\Delta H_{\bphi, \btheta}^{x, \zeta} )\Vert_{L^2([0,T])}^2 \right)
\end{align}
Note that the term $(\eta + \zeta)^2$ accounts for the fact that we are not solving exactly the dynamics.
\end{proof}

\end{document}